\documentclass[11pt]{article}
\usepackage{geometry} 
\usepackage{graphicx}
\usepackage{epstopdf}
\usepackage{amsmath}								
\usepackage{amssymb}
\usepackage{color}
\usepackage[dvipsnames,svgnames,table]{xcolor}
\usepackage[colorlinks=true,
            linkcolor=red,
            urlcolor=magenta,
            citecolor=blue]{hyperref}
\usepackage{caption}
\usepackage{float}
\usepackage{amsfonts}
\usepackage{amsthm}
\usepackage{hyperref}
\usepackage{authblk}
\usepackage{bbold}
\usepackage[ruled]{algorithm2e}
\usepackage{subfigure}

\newtheorem{theorem}{Theorem}[section]
\newtheorem{lemma}[theorem]{Lemma}
\newtheorem{prop}[theorem]{Proposition}
\newtheorem{defn}[theorem]{Definition}
\newtheorem{ass}[theorem]{Assumption}
\newtheorem{remark}[theorem]{Remark}
\newtheorem{cor}[theorem]{Corollary}

\usepackage{ulem}  

\newcommand{\norm}[1]{\left\lVert #1 \right \rVert}
\newcommand{\inner}[2]{\left\langle #1 , #2  \right\rangle}

\title{Mean-field optimal control with stochastic leaders}

\author[a,b]{Sebastian Zimper}
\author[c,b]{Ana Djurdjevac}
\author[d]{Carsten Hartmann}
\author[a,b]{Christof Schütte}
\author[a]{Nata\v sa Djurdjevac Conrad}
\affil[a]{Zuse Institute Berlin, Germany}
\affil[b]{Institut f\"ur Mathematik und Informatik, Freie Universit\"at Berlin, Berlin, Germany}
\affil[c]{University of Oxford, Mathematical Institute, Woodstock Road, Oxford OX2 6GG, UK}
\affil[d]{Brandenburgische Technische Universit\"at Cottbus-Senftenberg, Cottbus, Germany}

\begin{document}
\maketitle

\begin{abstract}
We consider interacting agent systems with a large number of stochastic agents influenced by a fixed number of external stochastic lead agents. Such settings arise, for example in models of opinion dynamics, where a small number of leaders can steer the behaviour of a large population of followers. In this context, we study a partial mean-field limit where the number of followers tends to infinity, while the number of leaders stays constant. The partial mean-field limit dynamics is then given by a McKean-Vlasov stochastic differential equation (SDE) for the followers, coupled to a controlled It\^o-SDE governing the dynamics of the lead agents. For a given cost functional that the lead agents seek to minimise, we show that the unique optimal control of the finite agent system converges to the optimal control of the limiting system. This establishes that the low-dimensional control of the partial (mean-field) system provides an effective approximation for controlling the high-dimensional finite agent system. In addition, we propose a stochastic gradient descent algorithm that can efficiently approximate the mean-field control. Our theoretical results are illustrated on opinion dynamics model with lead agents, where the control objective is to drive the followers to reach consensus.       
\end{abstract}

\section[]{Introduction}
\label{sec:Introduction}

Interacting agent and interacting particle models appear in a large variety of contexts, from classical examples in statistical physics such as the voter model \cite{holley1975ergodic} to more recent applications that describe collective motion in biological systems \cite{Vicsek2012}, 
systemic risk in  financial systems \cite{Garnier2013} or human behaviour in social systems \cite{Naldi2010,Djurdjevac2022}. From a mathematical perspective, interacting agent (or particle) systems are often studied in the context of mean-field theory, assuming that the agents (or particles) are indistinguishable. Under this assumption, the interaction between any agent (or particle) with the rest of the population can be described by an average interaction (``mean-field''), which is realised in the limit of infinitely many agents (particles). For an overview of this topic, we refer to the seminal work by Sznitman \cite{SznitmanTopics}. In many applications, it is desirable to control these systems such that some given objective is reached, e.g. in crowd control \cite{hughes2003flow} or in consensus-based optimisation \cite{pinnau2017consensus}. When the control is acting on only a few selected agents (or particles), which in the language of control theory corresponds to an underactuated control system, the assumption that all agents (particles) are indistinguishable is no longer valid. Nonetheless, it is still possible to consider the collection of uncontrolled agents (particles) as indistinguishable and treat them within a partial mean-field framework \cite{Fornasier2014,Helfmann2023}.

In this paper, we will follow the strategy of such partial mean-field limits, that give rise to sparse optimal controls by keeping the number of controlled agents (or particles)  constant, while letting the number of uncontrolled particles tend to infinity.  Specifically, we consider agent-based systems that consist of weakly interacting diffusion processes driven by uncorrelated Gaussian white noise acting on all agents, where only some of the agents are controlled. 
Throughout this paper, we will refer to the agents on which a control is exerted as \textit{leaders} and the remaining agents as \textit{followers}. 

The goal of our paper is two-fold: Firstly, we want to characterise the limit dynamics (including the optimal cost and control) as the number of followers goes to infinity, but the number of lead agents is kept constant, adopting results from \cite{Ascione2023,CarmonaZhu2016}. Secondly, we want to devise efficient numerical strategies to solve the corresponding optimal control problem for both finite and infinite agent populations, based on a stochastic gradient descent algorithm that exploits ideas put forward in \cite{Hartmann_2012,Lie2021}.  
Though closely related, our case differs from the situation of differential games in which all agents are controlled, but with possibly competing control objectives  \cite{Carmona2018}. Here the limiting problem is either a mean-field game or a feedback-controlled McKean-Vlasov SDE, depending on whether the agents compete against one another, or cooperate with each other \cite{Carmona2012}. While we consider stochastic control problems rather than games, our situation resembles mean-field games that comprise major and minor players \cite{CarmonaZhu2016}, with a few major players having the role of the lead agents and a large population of minor players being the followers. 

This research touches on ethically sensitive questions of influence in social systems, particularly regarding manipulation in digital communication spaces. Its results aim to provide a foundation for detecting, understanding, and counteracting attempts to manipulate public opinion and to support responsible governance and ethical use of data-driven social modelling.

\paragraph{Literature review and our contribution}

Results on propagation of chaos and mean-field limits for controlled or uncontrolled stochastic particle systems in the absence of leaders are by now classical; see e.g.  \cite{SznitmanTopics,Chaintron2022a,Chaintron2022b} and references therein. The limiting system is a so called McKean-Vlasov SDE, an SDE with coefficients which depend on the law of the process; for optimal control problems, the resulting McKean-Vlasov SDE is usually of feedback form where the control policy in general depends on the law of the process. Optimisation tasks of this form are sometimes referred to as mean-field optimal control problems, and we refer to  \cite{Bensoussan2013,Lauriere2014} for  formulations based on a dynamic programming principle or to \cite{Andersson2010,Carmona2015FBSDE} for the dual formulations in terms of a stochastic maximum principle. For such problems, it is not obvious that the optimisation step commutes with the mean-field limit, and to our knowledge the first paper that provides sufficient conditions under which the sequence of optimal controls for the finite particle system converges to the solution of the controlled McKean-Vlasov SDE is \cite{Lacker2017};  in \cite{Djete2022}, this analysis was later on extended to the case when common noise is present in the system dynamics.    

The majority of the literature on this topic is in the case when the control acts on every particle, i.e., in the absence of leaders. When leaders are present it was shown in \cite{Ascione2023} that if there is no noise in the equation for the leaders, then the sequence of optimal controls for the $N$ particle system $\Gamma$-converges to the optimal control of the limiting system as $N \to \infty$. A related setting involving infinitely many deterministic leaders was studied in \cite{Almi2023}. Propagation of chaos for stochastic lead dynamics under the assumptions that the drift coefficients are Lipschitz and bounded and the diffusion coefficients are constant was proved in \cite{CarmonaZhu2016}. Specifically, the results therein show that the optimal control for the limiting system gets arbitrarily close to minimising the cost functional of the finite system as $N \to \infty$. The proof is carried out for mean-field games with major and minor players, but the differential game set-up simplifies to a mean-field control problem for the lead agent when the minor players' strategies do not impact the dynamics; cf.~ \cite[Section 6]{CarmonaZhu2016}.  

In this work we extend the $\Gamma$-convergence result of \cite{Ascione2023} to the case of stochastic leaders when the state space of the particles is compact, exploiting ideas from \cite{CarmonaZhu2016} for bounded Lipschitz coefficients (with leaders). From this partial mean-field limit, we obtain a McKean-Vlasov SDE with coefficients that depend on the law of the process conditioned on the noise driving the leaders and that is coupled to an It\^o SDE for the lead agents.  An equivalent characterisation of the limit dynamics is in the form of a non-linear Fokker-Planck equation that governs the conditional law of the limiting system of followers and that is coupled to the SDE that governs the state of the lead agents. By adapting arguments in \cite{Kurtz1999} we show that the non-linear Fokker-Planck equation has a unique solution which has an $L^2$-density with respect to the Lebesgue measure. The partial mean-field limit of an interacting agent system with lead agents and followers is the first key contribution of this paper.

With regard to the numerical solution of the resulting high-dimensional optimal control problem, discretising the dynamic programming equation, a typically non-linear parabolic or elliptic partial differential equation (PDE), is usually out of question, even for a moderate number of agents; see \cite{achdou2020mean,kushnerBook} and the references therein. 
The same goes for the formulation of the optimal control problem based on the stochastic maximum principle, which leads to a pair of fully coupled forward-backward SDEs (FBSDEs) for the stochastic dynamics of the agents and the adjoint variable (which is typically the gradient of the associated value function), the numerical discretisation of which requires fully implicit time stepping schemes that can be challenging  \cite{chessari2023numerical}. We should mention that there are situations, in which the dynamic programming equation has a special form and can be represented in terms of a \textit{decoupled} FBSDE for the uncontrolled forward dynamics and the scalar value function (e.g.~\cite{hartmann2019variational,kebiri2019adaptive}), which can be discretised by explicit time stepping schemes, but these are rare cases. A viable alternative is the class of stochastic optimisation methods, such as the Robins-Monro algorithm \cite{harold2003stochastic}, least-squares Monte Carlo \cite{belomestny2010regression}, cross-entropy minimisation \cite{zhang2014applications}, or methods that are inspired by machine learning techniques and can handle potentially high-dimensional problems \cite{nusken2021solving}. 

In this work we adapt the results from \cite{Hartmann_2012,Lie2021} to give a numerical approximation to the optimal controls using a gradient descent algorithm for finding an optimal  control, based on explicit gradient computations and a representation of the feedback control in terms of finitely many ansatz functions. For finitely many agents, this scheme is straightforward to implement and not too computationally demanding. Moreover, and this is the second key contribution of this paper, we show that the algorithm generalises  to the partial mean-field framework and apply it to the case when the controlled SDE for the lead agents is coupled to a non-linear McKean-Vlasov-Fokker-Planck equation for the mean-field of the  uncontrolled agents. The algorithm belongs to the class of policy gradient algorithms \cite{peters2011policy}, with the specific that in our case the optimal control depends on both the state of the lead agents and the conditional law of the followers. We follow a first-discretise-then optimise strategy to apply the policy gradient algorithm to the semi-discretised SDE-PDE system for the leader and the mean-field of the followers (which then becomes an SDE-ODE system), for which the gradient can be explicitly computed, using Girsanov's theorem. 

As a demonstration we study the optimal control of the noisy Hegselmann-Krause model \cite{Wang2017} for a finite number of particles and in the mean-field limit. The Hegselmann-Krause model is a prototypical model from opinion dynamics, describing how the opinions of individuals evolve as they interact with each other \cite{Hegselmann2002}. We aim to bring the population to a consensus which has also been studied in \cite{Wongkaew2014} for the finite, deterministic model with leaders, and in \cite{Albi2015} and \cite{Bicego2025} for the Fokker Planck equation associated with the limiting systems of the deterministic and stochastic models without leaders, respectively. To the best of our knowledge this is the first time that the optimal control of the noisy Hegselmann-Krause model has been studied in the presence of leaders. The introduction of stochastic noise substantially alters both the analytical framework and the numerical methodology compared to the deterministic setting considered in \cite{Wongkaew2014}, as one must additionally account for fluctuations appearing in the dynamics.

\paragraph{Outline of the paper}

The rest of this paper is organised as follows. In Section \ref{sec:Ncontrol} we describe the stochastic optimal control for the finite particle system as well as the assumptions used throughout this paper. Section \ref{sec:MFLderivation} is devoted to the derivation of the mean-field limit. Next, in Section \ref{sec:MFLcontrol} we study the optimal control of the limiting system proving that the optimal controls of the finite system converge to it. The gradient descent algorithm to find an approximate solution of the optimal control for both the finite, and limiting system is discussed in \ref{sec:NumAlg}. This section also contains the derivation of the non-linear Fokker-Planck equation associated to the mean-field limit, and a result on the uniqueness of its solutions. Various numerical results on the optimal control of the finite and limiting systems for the Hegselmann-Krause model are presented in Section \ref{sec:Application}. Lastly, our conclusions and outlook are given in Section \ref{sec:Conclusion}.

\section[]{Optimal control for the $N$ particle system with leaders}
\label{sec:Ncontrol}

Let us consider a particle system consisting of $N \in \mathbb{N}$ identical followers and a leader which aims to minimise some given criteria. The state space of each follower and the leader is $\mathbb{R}^d$. For notational simplicity, the results are presented for a single leader. The extension to a fixed number of cooperating leaders who minimise a common cost is immediate, requiring only a corresponding enlargement of the leader state space. The dynamics of the followers, $X^N = \left( X^{1,N} , \dots, X^{N,N}  \right)$, and of the leader, $Y^N$ are given by
\begin{equation}
\label{eq:UncSys}
\begin{aligned}
    dX^{i,N}_t & = b(t,X_t^{i,N},Y_t^N,\mu_{X_t^N}) dt  + \sigma_X(t,X_t^{i,N},Y_t^N,\mu_{X_t^N}) dB_{t}^i , \quad i = 1, \dots , N   \\
    dY^N_t & = (c+u)(t,Y_t^N,\mu_{X_t^N}) dt  + \sigma_Y(t,Y_t^N,\mu_{X_t^N}) dB_{t}^Y ,
\end{aligned}
\end{equation}
where 
\begin{equation*}
    \mu_{X_t^N}  := \frac{1}{N} \sum_{j=1}^N \delta_{X_t^{j,N}} ,
\end{equation*}
is the empirical measure and $(B^i)_{i \leq N}$ and $B^Y$ are independent, $m$-dimensional Brownian motions. The coefficients $b$ and $\sigma_X$ describe the effect the other followers and the leader have on the dynamics of one follower, while $c$ and $\sigma_Y$ encode how the followers influence the leader. The control term $u$ is in feedback form and is chosen by the leader to minimise the cost functional 
\begin{equation}
    \label{eq:Cost}
    J^N(u) := \mathbb{E} \left[ \int_0^T r (Y_t^N,\mu_{X_t^N}) dt +  \frac{\lambda}{2} \int_0^T |u(t,Y_t^N,\mu_{X_t^N})|^2 dt  \right] ,
\end{equation}
consisting of a running cost $r$ and a `cost of control', the second term. Here $\lambda$ is a positive constant. The stochastic optimal control problem for a given class $U$ of admissible controls  is then to find the solution to the minimisation problem
\begin{equation}
    \label{eq:OC_problem}
    \min_{u \in U} J^N(u),
\end{equation}
such that the state dynamics \eqref{eq:UncSys} hold. If there exists a unique minimiser solving the optimal control problem, we shall denote it by $u_N^* \in U$.

Let us now give some of the mathematical formulation and assumptions used throughout this paper. For a Polish metric space $(\mathcal{E},d)$ we let $\mathcal{P}(\mathcal{E})$ denote the set of Borel probability measures on $\mathcal{E}$. For $p \geq 1$, we let $\mathcal{P}_p (\mathcal{E})$ be the subset of $\mathcal{P} (\mathcal{E})$ such that for all $\mu \in \mathcal{P}_p(\mathcal{E}) $ 
\begin{equation*}
    \int_\mathcal{E}  d(x,x_0)^p \mu(dx) < \infty ,
\end{equation*}
for any $x_0 \in \mathcal{E}$. The Wasserstein-$p$ distance, $W_{d , p}$,  between two elements $\mu$ and $v$ of $\mathcal{P}_p (\mathcal{E}) $ is defined by
\begin{equation*}
    W_{d , p} (\mu, v) := \inf_{\pi \in  \Pi(\mu,v)} \left( \int_{\mathcal{E}\times \mathcal{E}}  d(x,y)^p \pi(dx,dy)  \right)^{\frac{1}{p}} = \inf_{ \substack{ X \sim \mu \\ Y \sim v} } \mathbb{E} \left[ d(X,Y)^p   \right]^{\frac{1}{p}} ,
\end{equation*}
where $\Pi(\mu,v)$ is the set of all couplings of $\mu$ and $v$. We equip $\mathbb{R}^{d}$ with the Euclidean distance $|\cdot|$.
We work on a standard probability space $\left( \Omega, \mathcal{F}, \mathbb{P}  \right)$ on which we consider four, independent sources of randomness: the Brownian motions acting on the followers $(B^i)_{i \leq N}$ and leader $B^Y$, as well as the initial conditions $X_0^N = \left( X_0^{1,N} , \dots, X_0^{N,N}  \right)$ and $Y_0$. 

We shall make the following assumptions on the coefficients and cost functional.
\begin{ass}
    \label{ass:Lipschitz}
    For any $u \in U$ the functions
    \begin{equation*}
        \begin{aligned}
            (b, \sigma_X) &: [0,T] \times \mathbb{R}^{d} \times \mathbb{R}^{d} \times \mathcal{P}_2(\mathbb{R}^{d})  \to (\mathbb{R}^{d}, \mathbb{R}^{d \times m})  \\
            (c, u, \sigma_Y) &: [0,T] \times \mathbb{R}^{d} \times  \mathcal{P}_2(\mathbb{R}^{d})  \to (\mathbb{R}^{d}, \mathbb{R}^{d} ,\mathbb{R}^{d \times m})
        \end{aligned}
    \end{equation*}
    are measurable, uniformly continuous in $t \in [0,T]$ and globally Lipschitz continuous for the remaining variables, $\mathbb{R}^d$ being equipped with the Euclidean norm and $\mathcal{P}_2(\mathbb{R}^d)$ with the Wasserstein-$2$ distance.
\end{ass}
\begin{ass}
    \label{ass:Coeff_square_bound}
    For all $u \in U$, there exists $q \geq 2$, such that
    \begin{equation*}
        \mathbb{E} \left[  \int_0^T |b(t,0,0,\delta_0)|^q + |\sigma_X(t,0,0,\delta_0)|^q + |(c + u)(t,0,\delta_0)|^q    +|\sigma_Y(t,0,\delta_0)|^q dt \right]  < \infty .
    \end{equation*}
\end{ass}
\begin{ass}
    \label{ass:SDE_IC}
    The initial data $X^{N}_0  \in \mathbb{R}^{N \times d}$, and $Y_0\in \mathbb{R}^d$ are independent of each other and the $\sigma$-algebra generated by the Brownian motions $B^i$, $1, \dots, N$ and $B^Y$. Moreover, $(X_0^{i,N})_{i \leq N}$ are independent and identically distributed (i.i.d.) with common distribution $g_0 \in \mathcal{P}_2(\mathbb{R}^{d})$ and $\mathbb{E} \left[ |Y_0|^2  \right] < \infty$. 
\end{ass}

\begin{ass}
    \label{ass:costLipshtiz}
    There exists a constant $C\geq 0$, such that for any $y,y' \in \mathbb{R}^d$ and $\mu, v \in \mathcal{P}_2(\mathbb{R}^d)$, it holds
    \begin{equation*}
        \begin{aligned}
        & |r(y,\mu) - r(y',v)| \leq C \left( |y| + |y'| + M_2(\mu ) + M_2(v) \right) \times (   |y-y'| + W_2(\mu,v)   ),     
        \end{aligned}        
    \end{equation*}
    where $M_2(\mu) := \int |x|^2 \mu(dx)$.
\end{ass}
Assumptions \ref{ass:Lipschitz}, \ref{ass:Coeff_square_bound} and \ref{ass:SDE_IC} guarantee that the system \eqref{eq:UncSys} is well-posed for any $u\in U$, while Assumption \ref{ass:costLipshtiz} is required so that the cost functional \eqref{eq:Cost} is well-defined, see \cite{Carmona2018a}.

\section{Derivation of the mean-field limit}
\label{sec:MFLderivation}

The study of interacting particle systems becomes difficult as the number of particles in the system grows large. To circumvent this problem it is typical to derive a mean-field limit as $N \to \infty$, by proving a propagation of chaos result. An overview of the various methods to prove propagation of chaos can be found in \cite{SznitmanTopics,Chaintron2022a,Chaintron2022b} and the references therein.

As mentioned in the introduction the system \eqref{eq:UncSys} is non-standard, as the dynamics of the followers are coupled to those of the leader (for which we do not take a mean-field limit as we consider a total of one leader). It was shown in \cite{CarmonaZhu2016} that for the particle system \eqref{eq:UncSys} with constant diffusion coefficients the limiting system is a conditional McKean-Vlasov SDE. In this section we extend this result also to the case when the diffusion coefficient is non-constant.

We first establish some notation and recall essential  preliminaries. For $k \in \mathbb{N}$, the state space of $k$ followers $(\mathbb{R}^{d \times k}, d)$ is a Polish space, with the distance $d( \mathbf{x}^k, \mathbf{y}^k) := \frac{1}{k} \sum_{i=1}^k \left| x^i - y^i \right| $, where $ \mathbf{x}^k = (x^1 , \cdots , x^k)$. To simplify the notation we set $W_p := W_{d,p}$ and $\mathcal{W}_p := W_{W_p,p}$.

To derive the propagation of chaos result we first identify the limiting system. On the probability space $(\Omega,\mathcal{F},\mathbb{P})$ we consider another independent Brownian motion $B^X$, and initial conditions $(\bar{X}_0,\bar{Y}_0)$ and equip the probability space with the filtration $\mathbb{F} = (\mathcal{F}_t)_{t\geq 0}$, generated by $(B^X,B^Y,\bar{X}_0,\bar{Y}_0)$. Moreover, let $ (\mathcal{F}^Y_{t})_{t\geq0}$ be the filtration generated by the Brownian motion $B^Y$ and initial condition $Y_0$. For some final time $T > 0 $ the limiting system as $N \to \infty$ is identified as (see \cite{CarmonaZhu2016} and Theorem \ref{thm:PropChaos}) 
\begin{equation}
    \label{eq:BasicMcKeanSDE}
    \begin{aligned}
        d \bar{X}_t &= b(t,\bar{X}_t,\bar{Y}_t,g_t) dt + \sigma_X(t,\bar{X}_t,\bar{Y}_t,g_t) dB^X_{t} \\
        d \bar{Y}_t &= (c+u)(t,\bar{Y}_t,g_t) dt + \sigma_Y(t,\bar{Y}_t,g_t)  dB^Y_{t} \\
        g_t &= Law(\bar{X}_t | \mathcal{F}_{t}^Y)
    \end{aligned} 
\end{equation}  
where $B^X$ and $B^Y$ are independent $m$-dimensional Brownian motions. These types of systems are commonly referred to as conditional McKean-Vlasov systems \cite{Carmona2018a}. For square integrable initial data and under Assumptions \ref{ass:Lipschitz} and \ref{ass:Coeff_square_bound}, \eqref{eq:BasicMcKeanSDE} is the limiting system of \eqref{eq:UncSys}, in the sense that uniformly in $t$ the empirical measure $\mu_{X_t^N}$ converges to $g_t$ as $N \to \infty$ in the metric space $\mathcal{W}_2$. For a precise statement see Theorem \ref{thm:PropChaos}.

\begin{remark}
    The regular conditional distribution of $\bar{X}_t$ given $\mathcal{F}_{t}^Y$ exists, since the state space of the particles is Polish, see \cite[Theorem 1.1]{Carmona2018a}. This allows us to define $Law(\bar{X}_t | \mathcal{F}_{t}^Y)$ appropriately.
\end{remark}

\begin{defn}\label{def:sol}
    For any square integrable initial condition $(X_0,Y_0)$ taking values in $\mathbb{R}^d \times \mathbb{R}^d$, by a (strong) solution of \eqref{eq:BasicMcKeanSDE} we mean a continuous, $\mathbb{F}$-adapted process $ (\bar{X}_t,\bar{Y}_t)_{t \geq 0}$, such that for all $t \in [0,T]$, 
    \begin{equation*}
        \begin{aligned}
            \bar{X}_t &= X_0 + \int_0^t b(s,\bar{X}_s,\bar{Y}_s,g_s) ds + \int_0^t  \sigma_X(s,\bar{X}_s,\bar{Y}_s,g_s) dB^X_{s} \\
            \bar{Y}_t &= Y_0 + \int_0^t (c+u)(s,\bar{Y}_s,g_s) ds + \int_0^t \sigma_Y(s,\bar{Y}_s,g_s)  dB^Y_{s} \\
            g_s &= Law(\bar{X}_s | \mathcal{F}_{s}^Y) ,
        \end{aligned}
    \end{equation*}
    holds $\mathbb{P}$-a.s.
\end{defn}

We then have the following well-posedness result.

\begin{prop} \label{prop:well_posedness}
    If $\mathbb{E}[|\bar{X}_0 |^q + |\bar{Y}_0 |^q] < \infty$ for some $q \geq 2$ and Assumptions \ref{ass:Lipschitz} and \ref{ass:Coeff_square_bound} hold for the same $q$, then \eqref{eq:BasicMcKeanSDE} has a unique  strong solution $(\bar{X},\bar{Y})$ in the sense of Definition \ref{def:sol} such that
    \begin{equation*}
        \mathbb{E} \left[ \sup_{t\leq T} |\bar{X}_t |^q + \sup_{t\leq T} |\bar{Y}_t |^q \right] \leq C \mathbb{E} \left[  |\bar{X}_0 |^q + |\bar{Y}_0 |^q \right] <\infty ,
    \end{equation*}
    where C is some constant depending on $T$, $q$ and the Lipschitz constants from the assumptions.
\end{prop}

\begin{proof}
    The assumptions on the initial condition and Assumption \ref{ass:Coeff_square_bound} guarantee that all quantities are in the appropriate Banach space, while Assumption \ref{ass:Lipschitz} allows us to set up a Gr\"onwall inequality from which the mapping is shown to be a contraction. 
    Therefore, the lemma follows from an application of Banach's fixed point theorem to an appropriately defined solution mapping, for details see \cite[Proposition 7.2]{CarmonaZhu2016}. 
\end{proof}

Let us now setup our propagation of chaos result which we shall prove using a so-called synchronous coupling method, see Section 4.1.2 in \cite{Chaintron2022a}.
For some $N \in \mathbb{N}$ let $X^N := (X^{1,N}, \dots, X^{N,N}), Y^N$ be the solution of the finite particle system \eqref{eq:UncSys} and $\bar{X}^N = (\bar{X}^{1,N}, \dots, \bar{X}^{N,N}), \bar{Y}$ be the solution of
\begin{equation}
    \label{eq:McKeanSDE}
    \begin{aligned}
         \bar{X}^{i,N}_t &= X_0^{i,N} + \int_0^t b(t,\bar{X}^{i,N}_s,\bar{Y}_s,g_s) ds + \int_0^t \sigma_X(t,\bar{X}^{i,N}_s,\bar{Y}_s,g_s) dB^i_{s} \\
         \bar{Y}_t & =  Y_0 +  \int_0^t (c+u)(t,\bar{Y}_s,g_s) ds + \int_0^t \sigma_Y(t,\bar{Y}_s,g_s) dB_{s}^Y .
    \end{aligned} 
\end{equation} 

Note that for every $i=1,\dots, N$, $g_t = Law(\bar{X}_t | \mathcal{F}_{t}^Y) = Law(\bar{X}^{i,N}_t | \mathcal{F}^Y_{t}) = Law(\bar{X}^{1,N}_t | \mathcal{F}^Y_{t}) $, since $(\bar{X}^{i,N}_t)_t$ are $\mathcal{F}^Y_{t}$ conditionally i.i.d., for details see \cite[Proposition $7.4$]{CarmonaZhu2016}. Due to the conditional independence of $\bar{X}_t^{i,N}$, we have that $Law(\bar{X}_t^N | \mathcal{F}^Y_{t}) = g_t^{\otimes N}$ and we denote the conditional law of $X_t^N$, defined by \eqref{eq:UncSys}, by $g_t^N := Law(X_t^N | \mathcal{F}^Y_{t}) $. We aim to prove a propagation of chaos result showing that the difference between  $g_t^N$ and $g_t^{\otimes N}$, in a suitable metric space, convergence to zero as $N \to \infty$. Since $g_t^N$ and $g_t^{\otimes N}$ are both conditional laws, they are random measures and to prove their convergence we need to consider the space $\mathcal{P}(\mathcal{P}(\mathbb{R}^{d \times N}))$  which we equip with the Wasserstein-$p$ metric $\mathcal{W}_p$. Note that from the definition of the Wasserstein-$p$ metric, Jensen's inequality, and property of the conditional expectation we have
\begin{align*}
     \sup_{t \leq T} \mathcal{W}_p^p (Law(g_t^N), Law(g_t^{\otimes N}) ) \leq \mathbb{E} \left[ \frac{1}{N} \sum_{i=1}^N \sup_{t \leq T} |X_t^{i,N} - \bar{X}^{i,N}_t |^p \right] .
\end{align*}

Following the proof of \cite[Proposition $2.11$]{Carmona2018a} the processes $(X^{i,N}, \bar{X}^{i,N})$ are identically distributed and all terms in the sum of the last inequality have the same expectation. For the propagation of chaos results, it is therefore sufficient to prove that 
\begin{equation*}
    \mathbb{E} \left[ \sup_{t \leq T} |X_t^{1,N} - \bar{X}^{1,N}_t |^p \right] \xrightarrow[]{N \to \infty} 0 .
\end{equation*}
This formulation will allow us to use the SDE representations of systems \eqref{eq:UncSys} and \eqref{eq:McKeanSDE} in the synchronous coupling method. Proving the previous limit would therefore lead to the following propagation of chaos result.

\begin{theorem}\label{thm:PropChaos}
Given that Assumptions \ref{ass:Lipschitz}, \ref{ass:Coeff_square_bound} and \ref{ass:SDE_IC} are satisfied, let $X_t^{i,N}$, $Y_t^N$ be the solutions of \eqref{eq:UncSys} and $\bar{X}_t^{i,N}$, $\bar{Y}_t$ be the solutions of \eqref{eq:McKeanSDE}. Then, for each $T \geq 0$
\begin{align}
     \lim_{N \to \infty} \left( \mathbb{E} \left[ \sup_{t \leq T} |X_t^{1,N} - \bar{X}^{1,N}_t |^2 \right] +  \mathbb{E} \left[ \sup_{t \leq T} |Y_t^N - \bar{Y}_t|^2   \right]    \right) = 0  ,
\end{align}
and, therefore,
\begin{align}
    \lim_{N \to \infty} \left( \sup_{t \leq T} \mathcal{W}_2^2 (Law(g_t^N), Law(g_t^{\otimes N}) ) +  \mathbb{E} \left[ \sup_{t \leq T} |Y_t^N - \bar{Y}_t|^2   \right]    \right) = 0 .
\end{align}
\end{theorem}

The proof of this theorem is given in the Appendix \ref{sec:Appendix1}. The proof is similar to the one performed in \cite[Section 7]{CarmonaZhu2016} which we have extended to include a non-constant diffusion coefficient and random $Y_0$. Obtaining a result with more general conditions on the coefficients could be possible, if, instead of using a pathwise coupling method to prove convergence of the empirical measures, we use compactness methods, i.e. proving tightness of the laws, convergence to the solutions of the limiting system and then uniqueness of the limit. We would expect that this methodology would allow us to relax some of the assumptions on the coefficients. There is a close analogy with systems in which, instead of a leader, a common noise term acts on the dynamics of the followers, since in that case the mean--field limit also takes the form of a conditional McKean--Vlasov SDE. The literature on propagation of chaos for systems with common noise is relatively well-developed, see e.g. \cite{Hammersley2021, Nikolaev2025, CrowellThesis}, in contrast to the case of systems with leaders considered here. Consequently, one could potentially adapt these existing results to our setting. As this goes beyond the scope of the present work, we leave it for future investigation.

\section{Optimal control for the McKean-Vlasov system}
\label{sec:MFLcontrol}

Let us now consider the optimal control of the limiting system \eqref{eq:BasicMcKeanSDE} introduced in the previous section. As for the finite system, we consider the control $u$ to only act on a leader, where the aim is to find the minimizer of the cost functional
\begin{equation}
    \label{eq:McKeanCost}
    J(u) := \mathbb{E} \left[ \int_0^T r (\bar{Y}_t,g_t) dt +  \frac{\lambda}{2} \int_0^T |u(t,\bar{Y}_t,g_t)|^2 dt  \right] ,
\end{equation}
where $\mu_{X_t^N} $ has been replaced by $g_t$ in \eqref{eq:Cost}. The space of admissible controls $U$ for the mean-field optimal control is the same as that for the finite control problem \eqref{eq:OC_problem}, and we denote an optimal mean-field control of \eqref{eq:McKeanCost} by $u^*$.

\begin{remark}
    A necessary and sufficient maximum principle for the mean-field optimal control problem was established in \cite[pp. 557--560]{Carmona2018a}, where the optimal dynamics are characterised by a system of conditional McKean-Vlasov forward-backward SDEs \cite[eq. (7.20)]{Carmona2018a}. Thus, uniqueness of an optimal control could in principle be obtained from a suitable well-posedness result for this forward-backward system, together with uniqueness of the Hamiltonian minimiser. However, as far as we are aware, no such well-posedness result is currently available under the assumptions considered here.
\end{remark}

For a case of a constant diffusion coefficient and deterministic $Y_0$, in \cite{CarmonaZhu2016} it was shown that $u^*$ was an $\epsilon_N$ optimal control for the $N$-particle system, meaning that for all $u \in U$, $J^N(u^*) \leq J^N(u) + \epsilon_N$, with $\epsilon_N \to 0$ as $N \to \infty$. This result was actually proved for the more general setting of a mean-field game with major and minor players of which the stochastic optimal control problem of leaders and followers is a specific case. A stronger $\Gamma$-convergence result of $J^N$ to $J$ was proved in \cite{Ascione2023} for the case of $\sigma_Y = 0$. Because $g$ and $\bar{Y}$ are deterministic for $\sigma_Y = 0$ they can be shown to exist in compact spaces, so this $\Gamma$-convergence  result guarantees the convergence of the minimizers of the finite problem to the minimizer of the limiting problem. We adapt the results from \cite{Ascione2023} to show the  convergence of the minimizers when $\sigma_Y \neq 0$, the state space is compact and the admissible controls have the following form. 

\begin{ass}
    \label{ass:AdControl}
    The set of admissible controls $U$ consists of functions $u(t,\bar{Y}_t,g_t) = h(t) f(\bar{Y}_t,g_t) $ where we assume that the family of functions $h \in C ( [0,T]; K )$ is uniformly bounded and equicontinuous, with values in a compact set $K\subset \mathbb{R}^{d \times l}$, and $f \in C(\mathbb{R}^d \times \mathcal{P}_2(\mathbb{R}^d) ; \mathbb{R}^d)$ is Lipschitz and bounded in all arguments. Additionally, we assume that the running cost $r: \mathbb{R}^d \times  \mathcal{P}_2(\mathbb{R}^d) \to \mathbb{R}$ is uniformly continuous and bounded. 
\end{ass}

Let us note that although the specification that the dynamics occur on a compact state space is restrictive, this assumption is often made when studying the mean-field limit of interacting particle systems, see e.g. \cite{Bicego2025}. In particular for the Hegselmann-Krause model which we consider in Section \ref{sec:Application}, it is sensible to consider a bounded domain as opinions cannot have an arbitrarily large magnitude. We then obtain the following lemma which is proved in the Appendix \ref{sec:Appendix2}.

\begin{lemma}  
    \label{lemma:ConvMcKean}
    Let Assumptions \ref{ass:Lipschitz} and \ref{ass:Coeff_square_bound} hold and Assumption \ref{ass:AdControl} be fulfilled for a sequence of controls $(u_N)_{N \in \mathbb{N}} = (h_N f_N)_{N \in \mathbb{N}}$, where all of the elements of the sequences $(h_N)_{N \in \mathbb{N}}$ and $(f_N)_{N \in \mathbb{N}}$ are bounded above by the same constant and the elements of the latter sequence all have the same Lipschitz constant. Furthermore, suppose that $h_N \to h$ and $f_N \to f$ uniformly in $C ( [0,T]; \mathbb{R}^{d \times l} )$ and $C(\mathbb{R}^d \times \mathcal{P}_2(\mathbb{R}^d) ; \mathbb{R}^d)$, respectively, such that Assumption \ref{ass:AdControl} holds. 
    Let $(\bar{X}_N,\bar{Y}_N)$ correspond to the solution of the conditional McKean-Vlasov SDE \eqref{eq:BasicMcKeanSDE} with $u_N = h_N f_N$ and $(\bar{X},\bar{Y})$ to the solution with $u = h f$ with the same square-integrable initial data.\footnote{Note that the dependence on $N$ enters through the sequence of control $(u_N)_{N \in \mathbb{N}}$ and not through the number of agents.} Then 
    \begin{equation*}
        \lim_{N \to \infty} \mathbb{E} \left[ \sup_{t \leq T} W_2^2 (g_{N,t}, g_t) +  \sup_{t \leq T} \left| \bar{Y}_{N,t} - \bar{Y}_{t} \right|^2    \right] =  0 ,
    \end{equation*}
    where $g_{N,t}:= Law(\bar{X}_{N,t} | \mathcal{F}_t^Y)$ and $g_{t}:= Law(\bar{X}_{t} | \mathcal{F}_t^Y)$.
\end{lemma}

The above lemma yields a liminf inequality, which is essential for establishing $\Gamma$-convergence. A detailed proof is provided in Appendix  \ref{sec:Appendix2}.

\begin{lemma} \label{lem:liminf}
    If the same conditions as in Lemma \ref{lemma:ConvMcKean} are fulfilled, then 
    \begin{equation*}
        \liminf_{N \to \infty} J(u_N) \geq J(u) .
    \end{equation*}
\end{lemma}

For the optimal control problem of the particle system introduced in Section \ref{sec:Ncontrol}, we obtain the following theorem concerning the $\Gamma$-convergence of $J^N$ to $J$.

\begin{theorem}
    \label{thm:GammaConvergence}
    Under the conditions of Lemma \ref{lemma:ConvMcKean}, it holds that $J^N \xrightarrow{\Gamma} J$.
\end{theorem}

We defer the proof of Theorem \ref{thm:GammaConvergence} to the Appendix \ref{sec:Appendix2}. For a compact state space  $D \subset \mathbb{R}^d$, the space $\mathcal{P}_p(D)$ is also compact, leading to the following corollary. 

\begin{cor}
Let the state spaces of the followers and leader $D \subset \mathbb{R}^d$ be compact, and suppose that Assumptions \ref{ass:Lipschitz}, \ref{ass:Coeff_square_bound}, \ref{ass:SDE_IC}, \ref{ass:costLipshtiz}, and \ref{ass:AdControl} are satisfied. Then, a solution $u^*$ of the optimal control problem for the mean-field limit \eqref{eq:BasicMcKeanSDE} exists a.s. Moreover, the limit of any convergent subsequence of optimal controls for the finite problem \eqref{eq:OC_problem}, is an optimal control for the mean-field limit.
\end{cor}

\begin{proof}
    This proof consists of two steps. Firstly, we show that there exists a limit $u^*$ for a subsequence of finite-dimensional optimal controls $u_N^*$. Secondly, we prove that this $u^*$ is indeed an optimal control for the mean-field limit. 
    
    Note that because $K$ is a compact subset of $\mathbb{R}^{d \times l}$, any sequence $(h_i)_{i\in \mathbb{N}}$, $h_i \in C([0,T];K)$, has a subsequence  $h_{i,n}$ converging uniformly to some $h^*$. From the Assumption \ref{ass:AdControl} and the Arzela-Ascoli theorem we have that any sequence $(f_i)_{i\in \mathbb{N}}$, for $f_i \in C(D \times \mathcal{P}_2(D); \mathbb{R}^l)$ has a uniformly converging subsequence to $f^*$. Therefore, there exists $u^* \in U$ (where $U$ is as defined in Assumption \ref{ass:AdControl}), such that the sequence of minimizers $u_N^* = h_N^*f_N^*$ of the finite optimal control problem  has a subsequence (which we do not relabel) converging to $u^*$, $u_N^* \to u^*$. This proves the first step. Next we show that the second step holds due to the properties of $\Gamma$-convergence. Theorem \ref{thm:GammaConvergence} yields the liminf inequality, i.e., for every
    subsequence $(u_N)_{N \in \mathbb{N}}$ which converges to $u^*$ in $U$ we have that $\liminf_{N \to \infty} J^N(u_N) \geq J(u^*)$; as well as the fact that $ J(u) = \lim_{N \to \infty} J^N(u)$ for all $u$. Consequently, for any $u \in U$, we obtain that
    \begin{align*}
        J(u^*) \leq \liminf_{N \to \infty} J^N(u_N^*) \leq \limsup_{N \to \infty} J^N(u_N^*) \leq \lim_{N \to \infty} J^N(u) = J(u) ,
    \end{align*}
    where the last inequality holds because $u_N^*$ is, by definition, a minimiser of $J^N$. From this it follows that the second step of the proof holds and we can conclude.
\end{proof}

While Theorem \ref{thm:GammaConvergence} holds on a general state space, the stochastic nature of the limiting dynamics means that convergence of the optimal controls is only guaranteed when the state space is compact. One possible way to relax this assumption would be to work on $\mathbb{R}^d$ and replace compactness by suitable coercivity and moment estimates. This would amount to proving equicoercivity of the family of optimization problems in a topology on controls and measures compatible with the Wasserstein distance. Such an extension is technically nontrivial, since the Brownian noise in the leader dynamics prevents pathwise boundedness. Alternative approaches, which have successfully established convergence results for optimal controls of McKean–Vlasov SDEs with common noise, include the use of relaxed controls \cite{Djete2022}, and the analysis of convergence of the associated value functions, characterised by Hamilton–Jacobi–Bellman equations \cite{Cardaliaguet2023}. Extending these results to the setting considered in this work is, however, beyond the scope of this manuscript.

\section{Numerical method for approximating the optimal control}
\label{sec:NumAlg}

In practise it is often not feasible to solve the finite particle or mean-field stochastic optimal control problem when the space of admissible controls $U$ is infinite-dimensional, as in Assumption \ref{ass:AdControl}. Instead, the optimal control is approximated by choosing a set of permissible controls $U$ which is spanned by a finite collection $(e_k(\cdot))_{k=1}^n$ of linearly independent elements. This reduces the stochastic optimal control problems to finite-dimensional (stochastic) optimisation problems. In this section we will discuss how to solve this reduced problem first for the finite-particle system and then for the mean-field limit.

\subsection[]{Approximating the $N$-particle control}

For the optimal control problem \eqref{eq:OC_problem} restricting $U$ to elements spanned by some linearly independent $(e_k)_{k=1}^n$, for $e_k : [0,T] \times \mathbb{R}^{d} \times  \mathcal{P}_2(\mathbb{R}^{d}) \to \mathbb{R}^d$ such that Assumption \ref{ass:AdControl} holds,
reduces the problem to the finite-dimensional optimisation problem 
\begin{equation}
    \label{eq:ResOC}
    \min_{a\in \mathbb{R}^n} J^N(u^a), \quad u^a(\cdot) := \sum_{k=1}^n a_k e_k(\cdot)
\end{equation}
subject to the system dynamics \eqref{eq:UncSys}. A gradient descent based algorithm for finding the solution to the reduced optimisation method was proposed in \cite{Hartmann_2012} for a deterministic and finite stopping time. This method was extended in \cite{Lie2021} to random stopping times. We recall the following result from \cite{Lie2021} in which the change of measure from Girsanov's theorem was used to derive formulas for the ($n$-th order) Fr\'echet derivatives of the cost functionals when the control acts on each agent (i.e. $N=0$). 
\begin{ass}
    \label{ass:Invertible_sigma}
    The coefficient $\sigma_Y$ is invertible and admits a constant $\alpha >0$ such that for all $t \in [0,T]$, $ x, y \in \mathbb{R}^d$ and $\mu \in \mathcal{P}_2(\mathbb{R}^d)$,
    \begin{equation*}
         x^T \left(\sigma_Y \sigma_Y^T \right)^{-1}(t,y,\mu) x \leq \alpha^{-2} |x|^2  .
    \end{equation*} 
\end{ass}

\begin{theorem} \label{thm:LieOriginal} \textnormal{\textbf{\cite[Corollary 5.2.]{Lie2021}}}
    Consider the optimal control problem \eqref{eq:OC_problem} when $N=0$ in \eqref{eq:UncSys} and the coefficients $c, u$ and $\sigma_Y$ are independent of their third argument, i.e.,
    \begin{equation}
        \label{eq:NoFollowers}
        dY_t^N  = (c+u)(t,Y_t^N) dt   + \sigma_Y(t,Y_t^N) dB_{t}^Y ,
    \end{equation}
    and similarly for the cost
    \begin{equation}
    J^N(u) = \mathbb{E} \left[ \int_0^T r (Y_t^N) dt +  \frac{\lambda}{2} \int_0^T |u(t,Y_t^N)|^2 dt  \right].
    \end{equation}
    Let Assumptions \ref{ass:Lipschitz}, \ref{ass:Coeff_square_bound}, \ref{ass:SDE_IC}, \ref{ass:costLipshtiz} and \ref{ass:Invertible_sigma} hold for a set of permissible controls $U$ spanned by linearly independent $(e_k(\cdot))_{k=1}^n$, and assume that $\phi_T\left(Y\right) := \int_0^T r (Y_t^N) dt \in L^2\left(\mathbb{P}\right)$ and is not a.s. constant. Then $J^N$ is strictly convex and there exists at most one $a^* \in \mathbb{R}^n$ such that the first Fr\'echet derivative of $J^N$ vanishes at $u_N^{a^*} \in U$.  Moreover, the first and second order partial derivatives of $a \mapsto J^N(u^a)$ are given by 
    \begin{align}
        D_{a_k} J^N\left(u^a\right) & = \mathbb{E}^{u^a} \left[ \left(\phi_T + \lambda \sigma_Y^2 \left( \frac{1}{2} \left(M_T^{u^a}\right)^2 + M_T^{u^a} \right) \right) M_T^{e_k} \right] \label{eq:FCDER} \\
        D_{a_k a_l} J^N\left(u^a\right) & = \mathbb{E}^{u^a} \left[ \left(\phi_T + \lambda \sigma_Y^2 \left( \frac{1}{2} \left(M_T^{u^a}\right)^2 + 2 M_T^{u^a} +1 \right) \right) M_T^{e_k} M_T^{e_l} \right] , \label{eq:SCDER}
    \end{align}
    where superscript $u^a$ of $\mathbb{E}$ indicates that the expectation is taken when the control $u^a$ is applied to the dynamics \eqref{eq:NoFollowers} and
    \begin{equation*}
         M_T^{e_k} := \int_0^T \sigma_Y^{-1} e_k\left(Y_t^N\right)  dB_{t}^Y .
    \end{equation*}    
\end{theorem}
The explicit representations of the derivatives \eqref{eq:FCDER} and \eqref{eq:SCDER} can then be used in gradient-based method for finding $a^* \in \mathbb{R}^n$ such that the finite-dimensional optimisation problem \eqref{eq:ResOC} is solved for a given set of admissible controls \cite{Hartmann_2012, Lie2021}. This result can be adapted in a straightforward manner when $N \neq 0$, giving the following corollary.
\begin{cor}
    \label{Cor:LieAdapt}
    Consider the optimal control problem \eqref{eq:OC_problem} when $N \neq 0$ and the coefficients $b, \sigma_X, c, u, \sigma_Y$ and running cost $r$ depend on both the state of the leader $Y_t^N$ and the empirical measure of the followers $\mu_{X_t^N}$. Then the results of Theorem \ref{thm:LieOriginal} still hold when including $\mu_{X_t^N}$ as the second argument of $b, \sigma_X, c, u, \sigma_Y$ and $r$.
\end{cor}

\begin{proof}
    The result follows formally by a perturbation argument based on Girsanov's theorem. For given admissible controls $u$ and $w$, the perturbed control $u+\epsilon w$ with $\epsilon \geq 0$ is considered. Girsanov's theorem allows one to rewrite the cost functional under the perturbed dynamics as the expectation under the original control $u$, weighted by the stochastic exponential from the Radon-Nikodym derivative. Differentiating with respect to $\epsilon$ and evaluating at $\epsilon=0$ yields the desired variational identity.
    This argument can be made rigorous by following the approach in \cite{Lie2021}. To this end, rewrite the dynamics \eqref{eq:UncSys} as the single SDE for $Z_t = (X^N_t,Y_t^N)$,
    \begin{equation*}
        dZ_t = (A+u)(t,Z_t) dt + \sigma(t,Z_t) dB^Z_t ,
    \end{equation*}
    where $A(t,Z_t)$ collects the uncontrolled drift components of the followers and leader, the control $u$ takes values in $\left\{0 \right\}^{N\times d} \times \mathbb{R}^{d}$, $B^Z_t = (B^1_t, \dots, B^N_t, B_t^Y)$, and \begin{equation}
    \sigma = 
    \begin{pmatrix}
      \sigma_X & 0_{(Nd)d} \\
      0_{d (Nd)} & \sigma_Y
    \end{pmatrix} ,
    \end{equation}
    where $0_{(Nd)d}$ is the zero matrix with $Nd$ rows and $d$ columns. With this reformulation, the computations in \cite{Lie2021} apply directly, with the only modification that the inversion of diffusion coefficients is required solely for the controlled components, namely $\sigma_Y$.
\end{proof}

\begin{remark}
    \label{rem:NoInvert}
     Since there are no assumptions on the invertibility of $\sigma_X$, it could be identically zero, i.e. the dynamics of $X^{i,N}$ are governed by ordinary (not stochastic) differential equations. This observation will be relevant in the following subsection where we numerically approximate the optimal control of the non-linear Fokker-Planck equation coupled to an SDE which arises in the mean-field limit.
\end{remark}

\subsection{Approximating the mean-field control}

For the mean-field optimal control problem we would ideally like to apply a similar methodology as for the finite particle system. 
Such a derivation is indeed possible at the level of the conditional McKean–Vlasov system, as explained in Remark \ref{rem:Girsanov4Vlasov} below. Nevertheless, since the coupled PDE/SDE formulation is particularly well suited to the first-discretise-then-optimise strategy used below, we will stick to PDE/SDE formulation and study the optimal control problem for the coupled system  
\begin{equation}
\label{eq:PDE_SDE}
\begin{aligned}
        \partial_t g_t &= - \nabla_x \cdot ( b(t,x, \bar{Y}_t,g_t)g_t) + \frac{1}{2} \sum_{i,j = 1}^d \partial_{x_i} \partial_{x_j} (a_{ij}(t,x,\bar{Y}_t,g_t) g_t ) ,  \\[5pt]
        d\bar{Y}_t & = (c+u)(t,\bar{Y}_t,g_t) dt + \sigma_Y(t,\bar{Y}_t,g_t)dB^Y_{t} ,
\end{aligned}
\end{equation}
instead of the McKean-Vlasov system \eqref{eq:BasicMcKeanSDE}. Here $a := \sigma_X \sigma_X^T$. The PDE for $g$ is commonly known as a non-linear Fokker-Planck equation. Let us note that although $g$ is a solution to a PDE, it is still a stochastic process due to the coupling with $\bar{Y}$. In this subsection we show that for the solutions $(\bar{X},\bar{Y})$ of \eqref{eq:BasicMcKeanSDE}, the conditional law $Law(\bar{X}_t | \mathcal{F}_{t}^Y)$ and $\bar{Y}$ are the unique solution of \eqref{eq:PDE_SDE} (see Theorem \ref{thm:PDE_McKean_Equiv}). Since the optimal control problems we consider for the mean-field limit depend only on $g$ and $\overline{Y}$ we can therefore solve them by considering the PDE/SDE system \eqref{eq:PDE_SDE} instead of the McKean-Vlasov SDE \eqref{eq:BasicMcKeanSDE}. As such we will also discuss our numerical algorithm to solve the finitely based optimal control problem under the dynamics of \eqref{eq:PDE_SDE}. 
Throughout this section we denote the collection of all finite signed Borel measures on $\mathbb{R}^d$ by $\mathcal{M}(\mathbb{R}^d)$, and for a measurable function $\varphi$ and $\mu \in \mathcal{M}(\mathbb{R}^d)$ let $\inner{\varphi}{\mu}$ be the integral of a $\varphi$ with respect to $\mu$.

\begin{defn}
    \label{def:PDE_sol}
    A continuous, $(\mathcal{M}(\mathbb{R}^d),\mathbb{R}^d)$-valued, $\mathbb{F}$-adapted process $(g,\bar{Y})$ is said to be an analytically weak solution of \eqref{eq:PDE_SDE} for the initial data $(g_0,Y_0)$, where $g_0$ is fixed and $Y_0$ may be random, if, for any $t \in [0,T]$ and $\varphi \in C_b^2 (\mathbb{R}^d)$, a.s.,
    \begin{equation}
    \label{eq:WeakPDE}
    \begin{aligned}
        \int_{\mathbb{R}^d} \varphi(x) g_t(dx) = & \int_{\mathbb{R}^d} \varphi(x) g_0(dx) + \int_0^t \int_{\mathbb{R}^d} \nabla_x \varphi(x) \cdot b(s,x,\bar{Y}_s,g_s) g_s(dx) ds \\ 
        &+ \int_0^t \int_{\mathbb{R}^d} \frac{1}{2} \sum_{i,j = 1}^d a_{ij}(s,x,\bar{Y}_s,g_s)  \partial_{x_i} \partial_{x_j} \varphi(x) g_s(dx) ds ,
    \end{aligned}
    \end{equation} 
    and the SDE for $Y$ in \eqref{eq:PDE_SDE} holds with the given initial conditions.
\end{defn}

To begin let us show that for a solution to the conditional McKean-Vlasov system \eqref{eq:BasicMcKeanSDE}, $(\bar{X}, \bar{Y}) $, the conditional law $g_t = Law(\bar{X}_t | \mathcal{F}_{t}^Y)$ and dynamics of the leader $\bar{Y}$ are an analytically weak solution to \eqref{eq:PDE_SDE}. We recall the following technical lemma (which can be found in \cite[Lemma 5.1.]{CarmonaZhu2016} or \cite[Lemma B.3.1.]{HammersleyThesis}) which allows us to exchange conditional expectations and integration.

\begin{lemma}
    \label{lem:Cond_Fubini} \textnormal{\textbf{\cite[Lemma 5.1.]{CarmonaZhu2016}}}
    If $H$ is an $(\mathcal{F}_t)_{t\geq 0}$-progressively measurable, square integrable process, and $B^X$ and $B^Y$ are independent Brownian motions, then 
    \begin{align*}
        & \mathbb{E} \left[  \int_0^t H_s dB^Y_s \middle| \mathcal{F}_t^Y  \right] = \int_0^t \mathbb{E} \left[  H_s  \middle| \mathcal{F}^Y_s  \right] dB^Y_s , \\
        & \mathbb{E} \left[  \int_0^t H_s dB^X_s \middle| \mathcal{F}_t^Y  \right] = 0 , \\
        & \mathbb{E} \left[  \int_0^t H_s ds \middle| \mathcal{F}_t^Y  \right] = \int_0^t \mathbb{E} \left[  H_s  \middle| \mathcal{F}^Y_s  \right] ds .
    \end{align*}
\end{lemma}

\begin{lemma}
    \label{lem:SolPDE}
    Let Assumptions \ref{ass:Lipschitz}, \ref{ass:Coeff_square_bound} and \ref{ass:SDE_IC} hold, such that \eqref{eq:BasicMcKeanSDE} has a unique solution $(\bar{X},\bar{Y})$. Then $(g,\bar{Y})$, where $g_t = Law(\bar{X}_t | \mathcal{F}_{t}^Y)$, is a solution of \eqref{eq:PDE_SDE}.
\end{lemma}

\begin{proof}
    Given $(\bar{X},\bar{Y})$ we have by It\^{o}'s formula that for any $\varphi \in C_b^2 (\mathbb{R}^d)$
    \begin{equation*}
        \begin{aligned}
        \varphi(\bar{X}_t) = & \varphi(\bar{X}_0) + \int_0^t \nabla_x \varphi(\bar{X}_s) \cdot b(s,\bar{X}_s,\bar{Y}_s,g_s) + \frac{1}{2} \sum_{i,j = 1}^d a_{ij}(s,\bar{X}_s,\bar{Y}_s,g_s)  \partial_{x_i} \partial_{x_j} \varphi(\bar{X}_s) ds \\
        & + \int_0^t  \nabla_x \varphi(\bar{X}_s) \cdot \sigma_X(s,\bar{X}_s,\bar{Y}_s,g_s)  dB^X_{s} .
        \end{aligned}
    \end{equation*}
    Taking expectations conditional on $\mathcal{F}_t^Y$, applying Lemma \ref{lem:Cond_Fubini} and using that $B^X$ is independent of $Y$ gives
    \begin{equation*}
    \begin{aligned}
        \mathbb{E} \left[ \varphi(\bar{X}_t) \middle\vert\ \mathcal{F}_t^Y  \right]  = & \mathbb{E} \left[ \varphi(\bar{X}_0)  \right]  + \int_0^t \mathbb{E} \left[ \nabla_x \varphi(\bar{X}_s) \cdot b(s,\bar{X}_s,\bar{Y}_s,g_s)  \middle\vert\ \mathcal{F}^Y_{s}  \right] ds \\
         & + \int_0^t \mathbb{E} \left[ \frac{1}{2} \sum_{i,j = 1}^d a_{ij}(s,\bar{X}_s,\bar{Y}_s,g_s)  \partial_{x_i} \partial_{x_j} \varphi(\bar{X}_s) \middle\vert\ \mathcal{F}^Y_{s}  \right] ds .  
    \end{aligned}
    \end{equation*} 
    Since $g_t = Law(\bar{X}_t | \mathcal{F}_t^Y)$, we obtain \eqref{eq:WeakPDE}. It is then clear that $\bar{Y}$ is also a solution of the SDE component of \eqref{eq:PDE_SDE}.
\end{proof}

Although we have shown that a solution to the McKean-Vlasov SDE \eqref{eq:BasicMcKeanSDE} yields a solution to the associated PDE/SDE system, it remains to establish uniqueness of this solution. We prove uniqueness in Theorem \ref{thm:PDE_McKean_Equiv} by following the work of \cite{Kurtz1999}. In that work there is no leader present, but there is a source of common noise in the dynamics of the followers, and instead of a coupled PDE/SDE system, a stochastic PDE is studied. The proof strategy is to first freeze the non-linear arguments in \eqref{eq:PDE_SDE} for a given solution to obtain a linear PDE. By proving uniqueness for the linear PDE and using the uniqueness result for the McKean-Vlasov SDE the uniqueness of the non-linear PDE/SDE system \eqref{eq:PDE_SDE} then follows. 

For fixed $\mathcal{M}(\mathbb{R}^d)$ and $\mathbb{R}^d$ valued process $V$ and $Y$ we consider the linear PDE for $U(t)$ given by 
\begin{equation}
    \label{eq:FrozenPDE}
    \inner{\varphi}{U(t)} = \inner{\varphi}{U(0)} + \int_0^t \inner{\nabla_x \varphi(x) \cdot b_s + \frac{1}{2} \sum_{i,j = 1}^d a_{ij,s}  \partial_{x_i} \partial_{x_j} \varphi(x) }{U(s)} ds ,
\end{equation}
where
\begin{equation*}
    \begin{aligned}
        b_s = b(s,x,Y_s,V_s), \quad a_{ij,s} = a_{ij}(s,x,Y_s,V_s) .
    \end{aligned}
\end{equation*}
To show that this linear PDE has a unique solution we will transform a $\mathcal{M}(\mathbb{R}^d)$ valued process into a $L^2(\mathbb{R}^d)$ valued process. Taking some $v \in \mathcal{M}(\mathbb{R}^d)$ and $\delta >0$ define 
\begin{equation}
    (T_\delta v) := \int_{\mathbb{R}^d} G_\delta(x-y) v(dy),
\end{equation}
where $G_\delta(x) = (2 \pi \delta)^{-\frac{d}{2}} \exp\left( - |x|^2/2 \delta \right)$
is the heat kernel. The following assumption is necessary to derive uniform bounds on the frozen PDE \eqref{eq:FrozenPDE}.
\begin{ass}
    \label{ass:Bound}
    The coefficients $b$ and $\sigma_X$ are bounded.
\end{ass}

The following lemma, which characterizes solutions of the linear PDE, is proved in Appendix \ref{sec:Appendix3}.
\begin{lemma}
    \label{lem:UniBoundsLinear}
    Under Assumptions \ref{ass:Lipschitz}, \ref{ass:Coeff_square_bound} and \ref{ass:Bound} if $U$ is an $\mathcal{M}_+(\mathbb{R}^d)$ valued solution of \eqref{eq:FrozenPDE} for fixed $U(0) \in L^2(\mathbb{R}^d)$, then $U(t) \in L^2(\mathbb{R}^d)$ almost surely and for all $t \geq 0$
    \begin{equation}
        \label{eq:LinUniBound}
        \mathbb{E}[\norm{U(t)}_{L^2}^2  ] \leq C \norm{U(0)}_{L^2}^2   <  \infty   .    
    \end{equation}
    
\end{lemma}

\begin{cor}
    \label{cor:LinPDEunique}
    Under the same assumptions as Lemma \eqref{lem:UniBoundsLinear}, if, for the fixed initial data $U(0) >0$ and $U(0) \in L^2(\mathbb{R}^d)$, then \eqref{eq:FrozenPDE} has at most one $\mathcal{M}_+(\mathbb{R}^d)$ valued solution.
\end{cor}

\begin{proof}
Let us assume two solutions with the same initial data $U_1$ and $U_2$ exist. Because the PDE is linear we have that $U_3(t) = U_2(t) - U_1(t)$ is also a solution, with $U_3(0) = 0$ and $U_3(t)\in L^2(\mathbb{R}^d)$ a.s. From \eqref{eq:LinUniBound} we have
\begin{equation}
    \mathbb{E} \left[ \norm{U_3(t)}^2_{L^2}   \right] \leq  \norm{U_3(0)}^2_{L^2}   e^{Ct} = 0 ,
\end{equation}
and that therefore $U_1(t) = U_2(t)$ a.s. for all $t$.
\end{proof}

We can now show the uniqueness of the solution of \eqref{eq:PDE_SDE}.

\begin{theorem}
    \label{thm:PDE_McKean_Equiv}
    Suppose that \ref{ass:Lipschitz}, \ref{ass:Coeff_square_bound}, \ref{ass:SDE_IC} and \ref{ass:Bound} hold and additionally $g_0 \in L^2(\mathbb{R}^d)$. Then \eqref{eq:PDE_SDE} has a unique analytically weak solution $(g,\bar{Y})$ in the sense of Definition \ref{def:PDE_sol}, such that for all $t\in [0,T]$, $g_t \in L^2(\mathbb{R}^d)$ $\mathbb{P}$-a.s. 
\end{theorem}

\begin{proof}
    From the results of Lemma \ref{lem:SolPDE} we know that for the unique solution $(\bar{X},\bar{Y})$ of \eqref{eq:BasicMcKeanSDE}, letting $g_t:=Law(\bar{X}_t | \mathcal{F}_{t}^Y)$, $(g,\bar{Y})$, is a solution of \eqref{eq:PDE_SDE}. Let us therefore consider another solution of \eqref{eq:PDE_SDE}, $(\hat{g},\hat{Y})$. Define the stochastic process $\hat{X}$ as the unique solution to
    \begin{equation}
        \label{eq:McKeanFrozen}
        d\hat{X}_t = b(t,\hat{X}_t,\hat{Y}_t,\hat{g}_t)dt + \sigma_X(t,\hat{X}_t,\hat{Y}_t,\hat{g}_t) dB^X_t,
    \end{equation}
    and let $\tilde{g}_t := Law(\hat{X}_t|\mathcal{F}_t^Y)$. Then by an application of It\^{o}'s formula and Lemma \ref{lem:Cond_Fubini}, as in the proof of Lemma \ref{lem:SolPDE}, $\tilde{g}$ is a solution of
    \begin{equation}
        \label{eq:FrozeninProof}
        \begin{aligned}
        \inner{\varphi}{U(t)} = & \inner{\varphi}{U(0)} +  \int_0^t \inner{\nabla_x \varphi(x) \cdot b(s,x,\hat{Y}_s,\hat{g}_s) }{U(s)} ds \\
        & + \int_0^t \inner{ \frac{1}{2} \sum_{i,j = 1}^d a_{ij}(s,x,\hat{Y}_s,\hat{g}_s)  \partial_{x_i} \partial_{x_j} \varphi(x) }{U(s)} ds .
        \end{aligned}
    \end{equation}
    Moreover, by the definition of $(\hat{g},\hat{Y})$ as a solution of \eqref{eq:PDE_SDE}, $\hat{g}$ is also a solution of \eqref{eq:FrozeninProof}. We have however shown in Corollary \ref{cor:LinPDEunique} that \eqref{eq:FrozeninProof} has a unique solution, so $\hat{g} = \tilde{g}$ a.s. Substituting $Law(\hat{X}_t|\mathcal{F}_t^Y)=\tilde{g}_t$ for $\hat{g}$ in \eqref{eq:McKeanFrozen} we see that
    \begin{equation*}
        d\hat{X}_t = b(t,\hat{X}_t,\hat{Y}_t,Law(\hat{X}_t|\mathcal{F}_t^Y))dt + \sigma_X(t,\hat{X}_t,\hat{Y}_t,Law(\hat{X}_t|\mathcal{F}_t^Y)) dB^X_t,
    \end{equation*}
    i.e. $(\hat{X},\hat{Y})$ is a solution of \eqref{eq:BasicMcKeanSDE}. By Proposition \ref{prop:well_posedness} this solution is unique and therefore $(\hat{g},\hat{Y}) = (g,\bar{Y})$ a.s. 
\end{proof}

Since the cost functional $J$ \eqref{eq:McKeanCost} is only evaluated in terms of the state of the leader and the conditional law of the followers $(\bar{Y},g)$, by the results of Theorem \ref{thm:PDE_McKean_Equiv} we can compute these quantities and solve the mean-field control problem using the PDE/SDE system \eqref{eq:PDE_SDE} instead of the conditional McKean-Vlasov system \eqref{eq:BasicMcKeanSDE}. 
As we aim to solve these optimal control problems numerically, we discretise the spatial component of the PDE to obtain a finite system of ODEs coupled with the leader's SDE. For a uniform spatial grid $(x_i)_{i=1}^n$ on a fixed, finite domain $D$, the spatially discretised PDE/SDE system has the general form 
\begin{equation}
    \label{eq:Disc_PDE}   
    \begin{aligned}
            \Dot{g}^n(t) &= \mathcal{D}(t,g^n(t),Y_t) ,  \\
            dY_t & = (c+u)(t,Y_t,g^n(t)) dt  + \sigma_Y(t,Y_t,g^n(t))  dB^Y_{t} ,
    \end{aligned}
\end{equation}
where $g^n(t) = (g_i(t))_{i=1}^n$ is an approximation of $g_t$ at the grid-points $(x_i)_{i=1}^n$, and $\mathcal{D}$ depends on the specific discretisation scheme used. As \eqref{eq:Disc_PDE} is a finite system of ODEs, corresponding to the discretised PDE, coupled to an SDE on which the control acts, from the discussion in Remark \ref{rem:NoInvert} we can directly apply the results of Corollary \ref{Cor:LieAdapt} to solve optimal control problems constrained by these dynamics. Therefore, by first discretising the PDE we obtain a finite system for which we can apply a gradient descent scheme to minimise over a set of controls spanned by a finite basis.

\begin{remark}\label{rem:Girsanov4Vlasov}
 The derivatives of the mean-field cost functional can also be derived directly from the conditional McKean–Vlasov formulation by means of the Cameron–Martin–Girsanov theorem. The key observation is that the conditional law of the follower process given the leader filtration is preserved under the induced change of measure. This directly follows from Bayes’ formula, since the corresponding density process is measurable with respect to the leader filtration; alternatively, it can be established using uniqueness of the associated non-linear Fokker–Planck equation for the time evolution of the conditional law. As a consequence, the derivative formulas follow by the same arguments as in the classical SDE setting underlying Corollary 5.3. We do not pursue this alternative derivation here, for a complete treatment would require substantial additional analysis, and the coupled PDE/SDE formulation is more consistent with the overall approach in this article.
\end{remark}

\subsection{Choice of permissible controls}

For the simulations in Section \ref{sec:Application}, the admissible controls used to approximate the optimal control are chosen to be piecewise constant in time. This choice is mathematically justified by the result in \cite{Jakobsen2019}, which shows that optimal piecewise constant processes converge to the optimal control over all admissible controls, as the interval on which the processes are constant decreases to zero. We can therefore obtain better approximations by refining the time discretization, although this becomes computationally challenging as the time intervals become too small. The choice of piecewise constant controls also has a natural real-world interpretation in situations when a leader updates their strategy only at discrete time points, for example through periodic polling. However, applying this framework to real-world polling data would require addressing additional challenges, including noisy observations, partial information and irregular sampling times, which is beyond the scope of the present work.

To be consistent with the assumptions in the previous section, we regularise the piecewise constant in time controls, such that they are Lipschitz continuous. In particular, if the time-points in which the controls are adapted are $0 = t_1 < t_2 < \dots < t_m = T$, then the basis functions of the admissible control set, $e_k(t)$ for $k = 1, \dots ,m-1$, are given by 
\begin{align}
    \label{eq:FinColl}
    e_k(t) = 
    \begin{cases}
    (t - t_k)/c ,& \text{if } t \in [t_k, t_k + c) \\
    1 ,& \text{if } t \in [t_k + c, t_{k+1} -c ) \\
    (t_{k+1} -t)/c ,& \text{if } t \in [t_{k+1} -c , t_{k+1}) \\    
    0,              & \text{otherwise,}
\end{cases}
\end{align}
for some small constant $c=0.005$. The Newton's method for computing the optimal combination of these basis functions is outlined in Algorithm \ref{alg:one}.

\begin{algorithm}[tbh]
\caption{Newton's method iteration}\label{alg:one}
\KwData{initial guess for the optimal control $a^0$; initial system state $(X',Y')$; final time $T'$; tolerance $\epsilon$}
\KwResult{$(u,a^j)$ approximate solution to the optimal control problem \eqref{eq:ResOC} after $j$ iterations}
$j=0$\

\While{not converged}{
    $j = j +1$\;
    $\nabla_{a^{j-1}} J^N(a^{j-1} ) \gets$ \eqref{eq:FCDER} for system \eqref{eq:UncSys} with initial condition $(X',Y')$ for $T'$\;
    $H(a^{j-1}) \gets$ \eqref{eq:SCDER} \;
    $a^j \gets a^{j-1} - H(a^{j-1})^{-1} \nabla_{a^{j-1}} J^N(a^{j-1} )$\;
    $u \gets$ \eqref{eq:ResOC} for $u^{a^j}$ with \eqref{eq:FinColl}\;
    if $\norm{J^N(a^j) - J^N(a^{j-1})} < \epsilon$ then converged = True; 
}

\end{algorithm}

When computing the optimal control we first find the optimal strategy over the whole time horizon starting at the given initial conditions and then apply it until the time $t_2$, when the new constant strategy is applied. At $t_2$ we take note of the system state and find the optimal control over the remaining time horizon with the system state at $t_2$ as the new initial condition. This procedure is then repeated until the final time. In this way we guarantee that the optimal control computed is adapted to the noise at the time-points $t_i$, $i = 1, \dots, m$, when the control changes. This type of control is referred to as a discrete-time Markov control policy. We outline this process in Algorithm \ref{alg:two}. This algorithm can be used to compute optimal strategies for both the finite particle system and the discretized PDE/SDE mean-field system. Although its computational efficiency depends on the specific form of the interaction, we observe that, for the numerical results presented in the next section, the mean-field control is computationally more efficient than the particle system with $N=99$ followers.

\begin{algorithm}[tbh]
\caption{Discrete-time Markov control}\label{alg:two}
\KwData{initial guess for the optimal control $a_n^0$; initial system state $(X_0,Y_0)$; final time $T$; partition of the time interval $(I_k)_{k=1}^{n+1}$}
\KwResult{Approximate solution to the optimal control problem as a discrete-time Markov control}
$(X',Y') \gets (X_0,Y_0)$\;
\For{$k =  1 ,\dots , n-1 $}{
    $(u,a_k^*) \gets$ Algorithm \ref{alg:one} with initial guess $a_k^0$, $(X',Y')$ and final time $T-t_k$\;
    $(X',Y') \gets$ solution of \eqref{eq:UncSys} at $t=t_{k+1}$ for $Z_{t_k} = (X',Y')$ with control $u$\;
   $a_{k-1}^0 \gets a_k^*$ without its first entry\; 
}
\end{algorithm}

Our choice of finite basis is of course only one option to approximate the optimal control. Another popular choice is to use neural networks as approximating functions. For example, neural networks have been used to solve stochastic optimal control problems for, importance sampling \cite{RiberaBorrell2023}, stochastic reaction diffusion equations \cite{Stannat2023} and standard McKean-Vlasov control \cite{CarmonaLauriere2022} (without leaders).

\section{Application to the noisy Hegselmann-Krause model}
\label{sec:Application}

In this section we demonstrate the application of our optimal control method described in the previous sections to a model of opinion dynamics. To ensure that the theoretical assumptions are satisfied, we regularise the dynamics. Numerical experiments nevertheless show good qualitative behaviour even in situations where these assumptions are violated, suggesting that an extension of the theory may be possible. For a finite agent system, we consider the noisy Hegselmann-Krause model with only one leader \cite{Wang2017}. The classical Hegselmann-Krause model is a prototypical agent-based model describing the opinion evolution of a fixed agent population, where each agent continuously adjusts its opinion toward the average of those within a certain confidence bound given by the interaction radius $R$. The addition of noise accounts for random fluctuations in the process of opinion formation. The dynamics of the followers and the leader are given by
\begin{equation}
    \label{eq:NoisyHK}
    \begin{aligned}
    dX^{i,N}_t & = b(t,X_t^{i,N},Y_t^N,\mu_{X^N_t}) dt  + \sigma dB_{t}^i , \quad i = 1, \dots , N   \\
    dY^N_t & = u(t,Y_t^N,\mu_{X^N_t}) dt + \sigma dB_{t}^Y ,        
    \end{aligned}
\end{equation}
where the interaction between the followers and leader is specified by
\begin{equation*}
    \begin{aligned}
        b(t,x,y,\mu) = k \int_{\mathbb{T}} a \left(x' - x  \right)  \mu(dx') + k_L a \left(y - x \right) 
    \end{aligned}
\end{equation*}
for $k, k_L >0$. For the standard Hegselman-Krause model, each follower updates their opinion based on the opinions of their peers in the neighbourhood determined by the interaction radius $R$. In the numerical results that follow, the jump-discontinuities are regularised, as in \cite{Bicego2025}, by shifting the base by a small constant $c_r= 1/64$ such that 
\begin{equation*}
   a(r) = 
   \begin{cases}
        r,& \text{if } |r| \leq R \\
        r \, \left(R + c_r - |r| \right) / c_r ,& \text{if } R< |r| < R + c_r \\
        0,       & \text{otherwise} .
    \end{cases} 
\end{equation*}
The leader's opinion is independent of followers opinions and changes only in order to influence the followers by adapting the control $u$. The diffusion coefficient $\sigma$ is assumed to be a positive constant, representing the stochastic fluctuations in the opinion dynamics that here (in order to reduce the number of parameters) we assume to be the same for all followers and the leader.

In the left panel of Fig.~\ref{fig:HKsim} we show a simulation of the Hegselmann-Krause model in the time interval $t \in [0,1]$ for $N=99$ followers (dark-blue lines) without a leader present. The histograms of the follower's opinions at the initial (red) and final (black) times are presented in the right panel. We set the system parameters to $k= 10$, $\sigma = 0.05$ and $R=0.15$. For the state space we consider the flat unit torus $\mathbb{T}$. A comparison of the model's qualitative behaviour under different boundary conditions, including periodic and no-flux conditions, is provided in \cite{Goddard2022}. The initial data is sampled from the sum of two von Mises distributions on the unit torus, $f(x|0.65,4) + f(x|0.25,8)$, where $f(x|\mu, \kappa) := Z_\kappa^{-1} \exp\{\kappa \cos \left( x -  2 \pi \mu  \right)\}$ and $Z_\kappa = \int_\mathbb{T} \exp\{\kappa \cos \left( x \right)\} dx$. The choice of this initial condition corresponds to two clusters of the population. Unless otherwise stated we will use the same parameters and initial condition for all the simulations presented in this section. From this simulation we see that the two clusters become more distinct as the system evolves.

Generally, a desirable outcome is for the entire population to form a consensus, i.e., for all agents to have a similar opinion. It is well known that, for a sufficiently small diffusion coefficient (or alternatively high interaction strength) the Hegselmann-Krause model will form several clusters which will eventually merge and form a consensus (see the discussions in \cite{Wang2017,Garnier2016,Gerber2026}). However, we are interested in reaching a consensus in a short, finite time horizon $T$, rather than asymptotically. To formalize this goal, we define the cost as
\begin{equation*}
    \begin{aligned}
        J^N(u) := \mathbb{E} \left[ \int_0^T \int_{\mathbb{T}} |Y_t^N - x'|^2 \mu_{X^N_t}(dx') + \frac{\lambda}{2}  |u(t,Y_t^N,\mu_{X^N_t})|^2 dt  \right ],
    \end{aligned}
\end{equation*}
where the leader aims to minimise the distance to all followers, thereby bringing the system to a consensus state. 

\begin{figure}[tbh]
    \centering
    \includegraphics[width=0.75\linewidth]{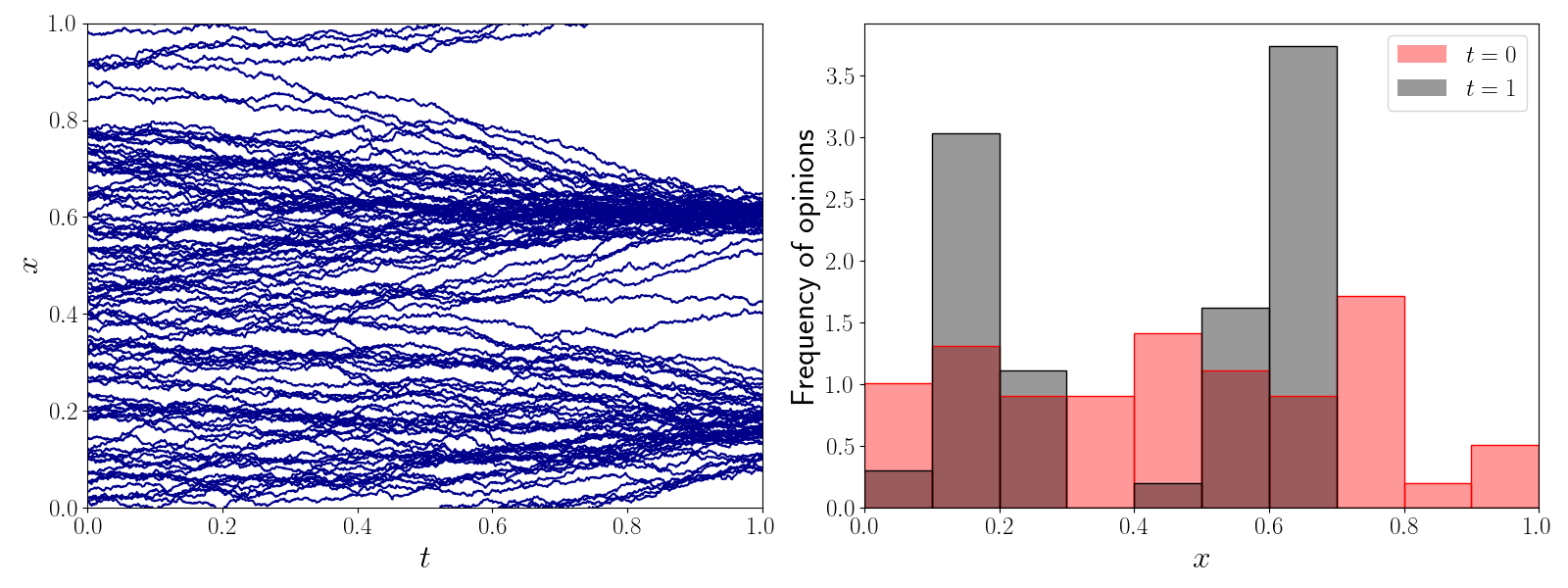}
    \caption{Left: Simulation of the noisy Hegselmann-Krause model for $N=99$ followers (dark-blue lines) without a leader present. Right: Histograms of the follower's opinions at $t=0$ (red) and $t=1$ (black).}
    \label{fig:HKsim}
\end{figure}

We implement Algorithm \ref{alg:one} to compute the optimal offline control for $\lambda = 0.01$ and five evenly spaced intervals over the total time interval $t \in [0,1]$ ($m=6$). The initial position of the leader is set to $Y_0 = 0.8$ and $k_L = 5$. A total of $M=10^4$ realisations were used in a Monte-Carlo scheme to compute the expectations (we found that further increasing the number of Monte-Carlo simulations had no significant effect on the results). 

The overall accuracy of the numerical approximation of the particle  system is affected by time discretisation and Monte Carlo  errors of the cost functional, its gradient, and its Hessian which, for the situation at hand, shows the usual $\mathcal{O}(\Delta t + M^{-1/2})$ decay; cf.~\cite{bossy1996convergence}. Note that the Euler-Maruyama method for our SDE with additive noise has both weak and strong order of convergence 1.  Assuming that the regression error due to the approximation of the optimal control by piecewise constant controls is sufficiently small, Newton's methods enjoys the usual quadratic convergence behaviour \cite{gobet2022newton}; the same goes for the positivity-preserving finite difference scheme for the PDE/SDE system mentioned below. We refrain from going into further details here and leave the numerical analysis for future work.

In the left panel of Fig.~\ref{fig:m5_iterations} the successive updates of the control vector $u^i$ are shown for $200$ total iterations of the gradient descent based algorithm. Since the control is a vector we show all five components at each iteration. The total cost of the successive controls $J^N(u^i)$, is presented in the right panel of Fig.~\ref{fig:m5_iterations}. We can see that the first two components of the control undergo the greatest change in the first $150$ iterations after which further iterations do not have a large impact. For the cost, we see a similar behaviour as there is initially a great decrease in the cost, which levels out as we approach the optimal control. 

\begin{figure}[tbh]
    \centering
    \includegraphics[width=0.75\linewidth]{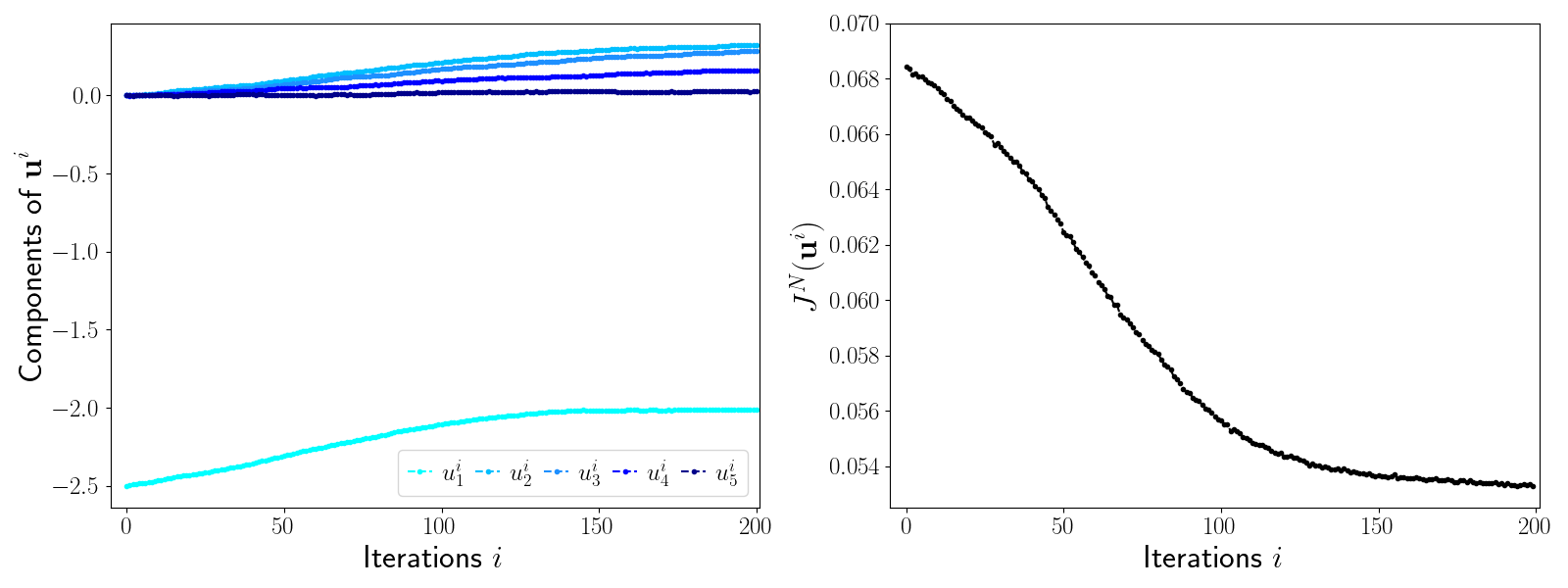}
    \caption{Results obtained by implementing the gradient descent based on Algorithm \ref{alg:one} for the noisy Hegselmann-Krause model with one leader and $99$ followers. The optimal control is approximated by functions which are piecewise constant over $m=5$ evenly space intervals \eqref{eq:FinColl} on the time interval $t \in [0,1]$. Left: The five components of the control $u^i$ for each iteration $i$ of the minimisation scheme. The subscript $u_k^i$ denotes the $k$-th component of $u^i$. Right: The total cost $J^N(u^i)$ for each control shown in left panel.}
    \label{fig:m5_iterations}
\end{figure}

We compute the dynamics for the optimal discrete-time Markov control using Algorithm \ref{alg:two}. The time-evolution of the optimally controlled system is shown in the left panel of Fig.~\ref{fig:m5_results}. The red line corresponds to the leader's dynamics and the black, dashed, vertical lines indicate when the piecewise constant control changes. Histograms of the follower's initial and final opinions are presented in the right panel of Fig.~\ref{fig:m5_results}. Comparing Figs.~\ref{fig:HKsim} and \ref{fig:m5_results} we can see that the leader succeeds in bringing the majority of followers to a consensus, exemplified by the single peak of the histogram at the final time.

\begin{figure}[tbh]
    \centering
    \includegraphics[width=0.75\linewidth]{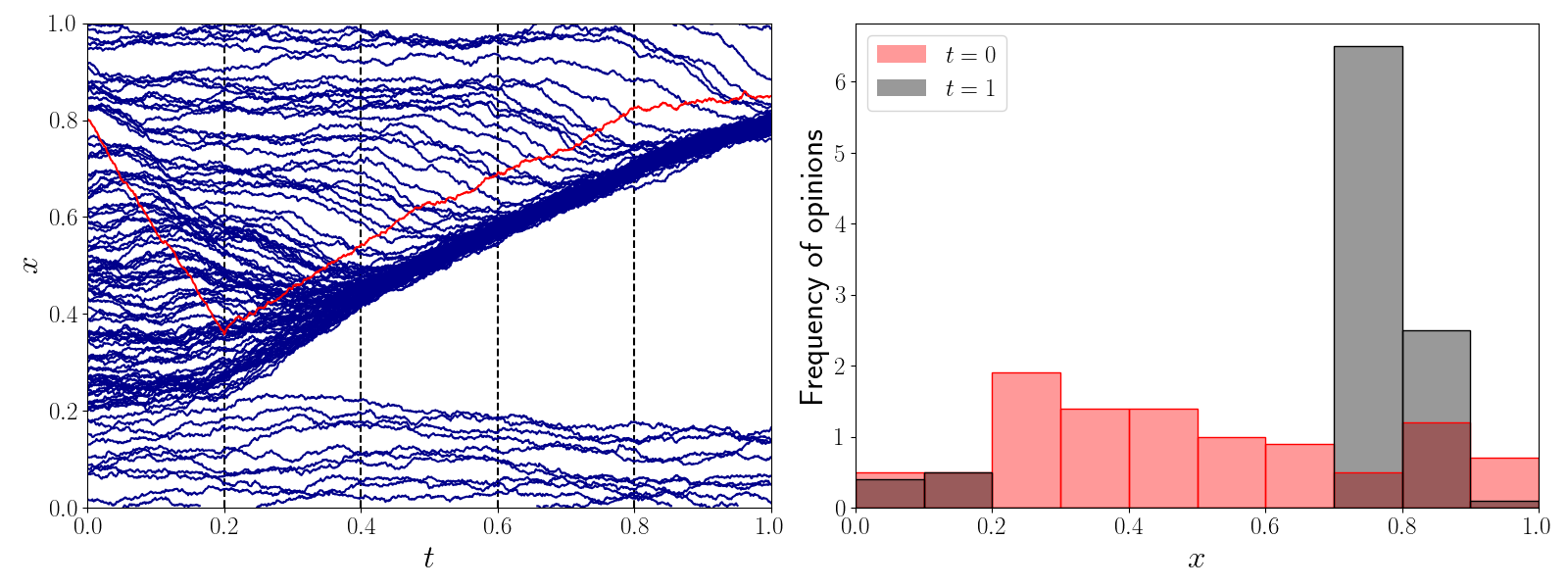}
    \caption{Left: System dynamics under the optimal discrete-time Markov control computed using Algorithm \ref{alg:two}, for $N=99$ followers (dark-blue lines) and one leader (red line). The black, dashed, vertical lines indicate when the piecewise constant control changes. Right: Histograms of the follower's opinions at $t=0$ (red) and $t=1$ (black).} 
    \label{fig:m5_results}
\end{figure}

So far, we have only solved the optimal control problem for relatively weak noise, $\sigma = 0.05$, as it means that we can verify our scheme without worrying about large fluctuations. To demonstrate that our method also works when the fluctuations due to the noise are greater, we show in Fig.~\ref{fig:HK_additional} the optimal dynamics for $\sigma=0.1$ (top-row) and $\sigma=0.2$ (bottom-row). The remaining system parameters are the same as before. By comparing the histograms at $t=1$ presented in Figs.~\ref{fig:m5_results} and \ref{fig:HK_additional}, we can see that as the noise increases the final cluster contains less agents and is more spread out. Nonetheless, the leader is still able to bring the system close to a consensus. 

\begin{figure}[tbh]
    \centering
    \includegraphics[width=0.75\linewidth]{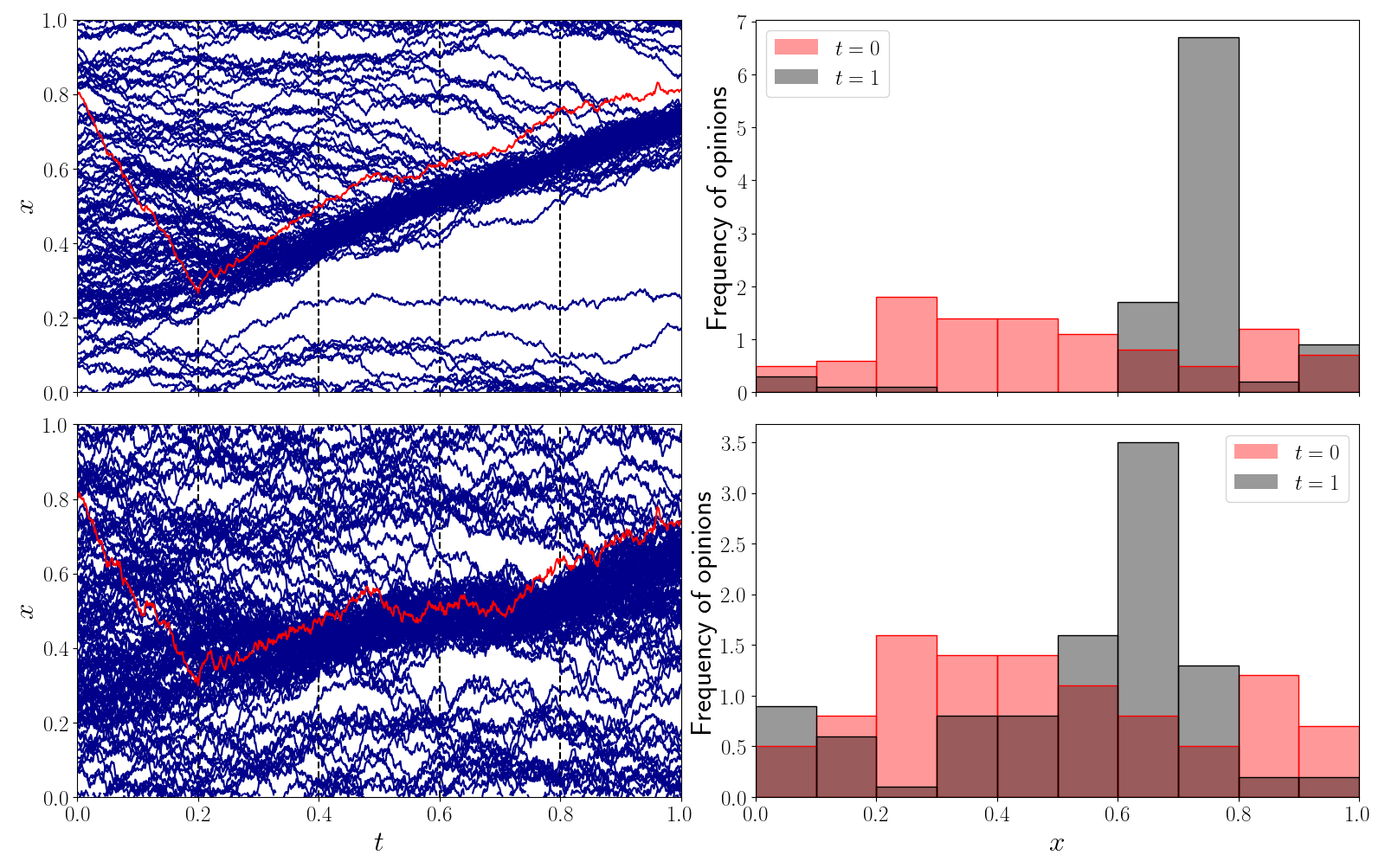}
    \caption{Similar to  Fig.~\ref{fig:m5_results}, but for $\sigma = 0.1$ (top-row) and $\sigma=0.2$ (bottom-row).}
    \label{fig:HK_additional}
\end{figure}

The PDE/SDE system for the mean-field limit of the Hegselmann-Krause model with a leader is 
\begin{equation*}
\begin{aligned}
        \partial_t g_t &= - \nabla_x ( b(t,x, \bar{Y}_t,g_t)g_t) + \frac{\sigma^2}{2} \Delta_x  g_t  ,  \\[5pt]
        d\bar{Y}_t & = u(t,\bar{Y}_t,g_t) dt + \sigma dB^Y_{t} ,
\end{aligned}
\end{equation*}
and the associated cost for the mean-field control problem is
\begin{equation*}
    J(u) := \mathbb{E} \left[ \int_0^T \int_{\mathbb{T}} \norm{\bar{Y}_t - x'}^2 g_t(dx') + \frac{\lambda}{2}  |u(t,\bar{Y}_t,g_t)|^2 dt  \right ] .
\end{equation*}

To solve the discretised PDE/SDE system, that in general is given by \eqref{eq:Disc_PDE}, we use the numerical scheme for non-linear Fokker-Planck equations introduced in \cite{Pareschi2018}. This scheme is second order accurate in space and first order accurate in time (cf.~\cite{duan2021structure}); if the PDE/SDE system is considered an approximation of the $N$-particle system, another $\mathcal{Q}(N^{-1/2})$ error for the approximation of the PDE solution by the empirical measure comes into play \cite{bossy1996convergence}. The finite-difference scheme of \cite{Pareschi2018} preserves non-negativity if the step-sizes of the time $\Delta t$ and spatial $\Delta x$ discretisations are chosen such that
\begin{equation*}
    \Delta t \leq \frac{\Delta x^2}{2\Delta x (k +k_L)R  + \sigma^2} \,.
\end{equation*}
As the state space of particle system in the previous subsection was selected to be the unit torus, the domain of the PDE is $[0,1]$ with periodic boundary conditions. We discretise this spatial domain into $n=64$ evenly spaced points. As discussed at the end of Section \ref{sec:MFLcontrol} the ODEs of the discretised PDE are coupled to the SDE of the leader allowing us to apply the same method as in the finite particle case to solve the optimal control problem. We again restrict ourselves to controls which are piecewise constant in time.

The system parameters are the same as for the results presented in Fig.~\ref{fig:m5_results} and $m=5$ evenly spaced intervals are used for the piecewise constant control. The initial data for the PDE is given by the sum of two von Mises distributions on the unit torus, $g_0(x) = f(x|0.65,4) + f(x|0.25,8)$ and initial position of the leader is $Y_0 = 0.8$ as before. The system dynamics when the optimal discrete-time Markov control is implemented is shown in left panel of Fig.~\ref{fig:CC} with $g_0(x)$ and $g_T(x)$ shown in the right panel. We can see that the leader is able to bring about a consensus, characterised by the single peak of $g_T(x)$. 

\begin{figure}[tbh]
    \centering
    \includegraphics[width=0.75\linewidth]{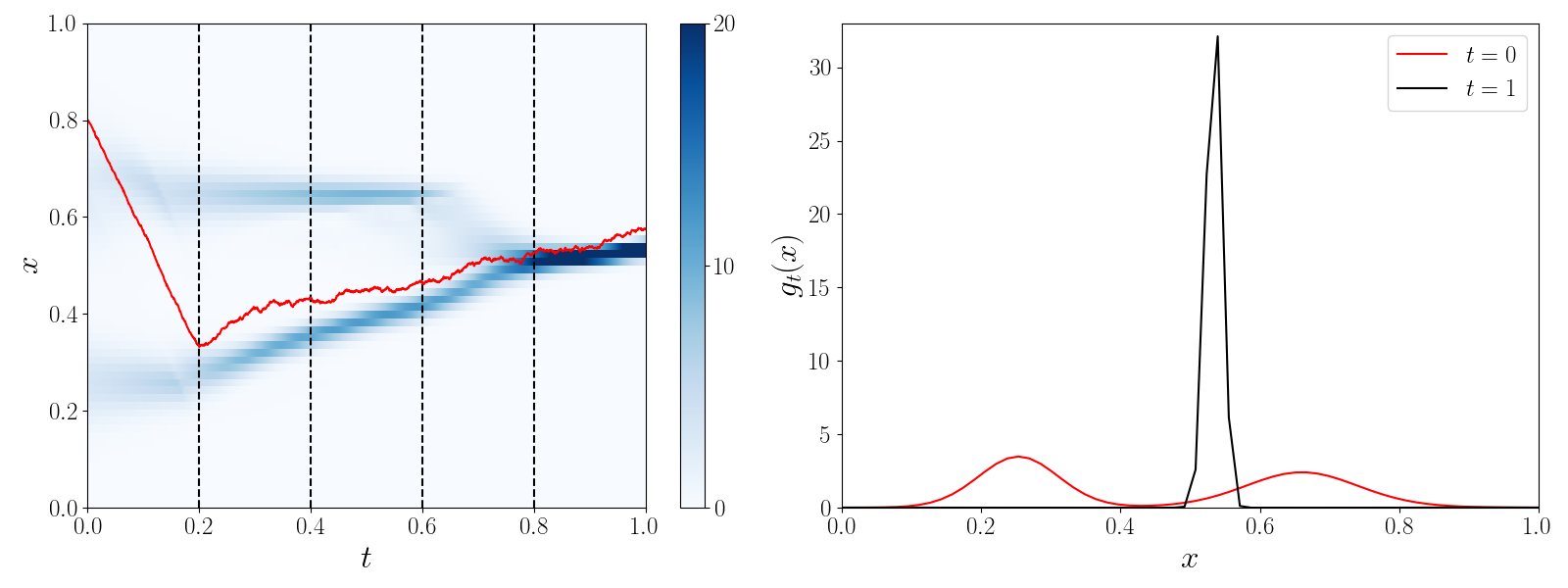}
    \caption{Left: System dynamics of the limiting non-linear Fokker-Planck equation coupled to the SDE of the leader when the optimal discrete-time Markov control is applied. The position of the leader is indicated by the red line and the density of followers $g_t(x)$ coloured according to the heatmap on the right of the plots. Right: The profile of $g_t(x)$ at the initial (red) and final (black) times.}
    \label{fig:CC}
\end{figure}

\section[]{Conclusion and outlook}
\label{sec:Conclusion}

In this work we studied the stochastic optimal control of finite particle (agent) systems, as well as their mean-field limit, through the intervention of a leader. We generalised existing theoretical results on the derivation of the mean-field limit as well as the connection between finite and mean-field optimal controls and developed a numerical method for computing the optimal control in the mean-field limit.

The mean-field limit is a conditional McKean-Vlasov SDE coupled to the SDE of the leader, which we showed in a propagation of chaos proof, for non-constant diffusion coefficients. The conditional law of the solution of the conditional McKean-Vlasov SDE was then proved to be the unique solution of a non-linear Fokker-Planck equation coupled to an SDE. Concluding our analysis we showed that when the state space of the particles is compact and the leader is driven by an independent Brownian motion, the optimal control of the finite system converges to that of the mean-field limit. 

The method we presented to approximate the optimal mean-field control consists of first discretising (in space) the non-linear Fokker-Planck equation to obtain a system of ODEs coupled to an SDE, and then applying a gradient descent-based minimisation scheme to find an optimal, finitely based, control. We utilised this method to find the optimal discrete-time Markov control for a leader to bring the noisy Hegselmann-Krause model to a consensus. While we only used piecewise constant basis functions, we expect that our method will also work for a different choice of finitely based approximation, e.g., neural networks.

There are several different directions for future work. Firstly, we aim to give a convergence analysis of the numerical scheme we have developed that would give more precise conditions on how well the optimal control is approximated using a given basis, spatial discretisation, number of Monte-Carlo simulations and gradient evaluations. Secondly, it is well known that to represent the fluctuations in large (but still finite) particle systems it is necessary to study a Dean-Kawasaki type SPDE and not the Fokker-Planck equation \cite{Helfmann2021}. While  there is little theoretical work on the optimal control of Dean-Kawasaki type SPDEs, it is interesting to study the problem from a numerical perspective by first discretising the SPDE, similar to what we have done for the PDE. In this case the numerical scheme used to guarantee positivity of the numerical solution is of particular interest, see e.g. \cite{Djurdjevac2025}. Lastly, an explicit expression for the Fr\'echet derivatives, with respect to a perturbation of the drift, of the expected value of functionals depending on the solution of conditional McKean-Vlasov SDEs remains open. 

The study of opinion dynamics and control inevitably touches on ethically sensitive questions, particularly regarding influence and manipulation in digital communication spaces. While our work is motivated by mathematical and computational insights, its implications extend to understanding how collective opinions can be shaped by a small number of agents or external drivers. We hope that these findings will contribute to developing transparent frameworks for detecting, analysing, and ultimately mitigating attempts to steer or distort public opinion. By offering a quantitative foundation for such analyses, this research seeks to support responsible governance and ethical use of data-driven social modelling.\\

{\bfseries Funding.} This work has been partially funded by the Deutsche Forschungsgemeinschaft (DFG, German Research Foundation) under Germany´s Excellence Strategy – The Berlin Mathematics Research Center MATH+ (EXC-2046/1, EXC-2046/2, project ID: 390685689). The research of CH has been partially funded by the German Federal Government, the Federal Ministry of Education and Research and the State of Brandenburg within the framework of the joint project EIZ: Energy Innovation Center (project numbers 85056897 and 03SF0693A).

\bibliographystyle{abbrv}
\bibliography{ref}

\section{Appendix}

\subsection{Proofs for Section \ref{sec:MFLderivation}}
\label{sec:Appendix1}

To begin let us recall the following two theorems which give the convergence of the empirical measure of $N$ i.i.d. particles to the common distribution as $N \to \infty$. The first follows from the Glivenko-Cantelli theorem, see \cite[Section 5.1.2]{Carmona2018}, and the second gives explicit rates of convergence.
\begin{theorem}
    \label{thm:GlivenkoCantelli}
    For a given i.i.d. sequence of $\mathbb{R}^d$-valued random variables $(X_n)_{n\in \mathbb{N}}$ with a common distribution $\mu \in \mathcal{P}_2(\mathbb{R}^d)$, $\lim_{N \to \infty} \mathbb{E} \left[ W_2(\mu_{X^N},\mu)^2  \right] = 0$, where $\mu_{X^N} := N^{-1} \sum_{i=1}^N \delta_{X_i}$ is the empirical measure of $(X_n)_{n \leq N }$.
\end{theorem}
\begin{theorem}[Theorem $5.8$ \cite{Carmona2018}]
    \label{thm:ExplicitRates}
    Given an i.i.d. sequence of $\mathbb{R}^d$-valued random variables $(X_n)_{n\in \mathbb{N}}$ with a common distribution $\mu \in \mathcal{P}_q(\mathbb{R}^d)$ for $q>4$, then for any $d \geq 1$ there exists a constant $C=C(d,q,M_q(\mu))$ such that for all $N \geq 2$,
    \begin{equation*}
        \mathbb{E} \left[ W_2(\mu_{X^N},\mu)^2  \right] \leq \epsilon_N := C 
        \begin{cases}
            N^{- \frac{1}{2}}, \quad & \text{if  $d< 4$}, \\
            N^{- \frac{1}{2}} \log N, \quad &\text{if  $d= 4$}, \\
            N^{- \frac{2}{d}}, \quad &\text{if  $d> 4$}.
        \end{cases}
    \end{equation*}
\end{theorem}
To prove Theorem \ref{thm:PropChaos} we will use the following probabilistic setup from \cite[Section 2.1.3]{Carmona2018a} which will let us separate the effects of the noise on the leader and followers. We begin with two complete probability spaces $(\Omega^0, \mathcal{F}^0,\mathbb{P}^0)$ and $(\Omega^1, \mathcal{F}^1,\mathbb{P}^1)$, respectively endowed with the right-continuous and complete filtrations $\mathbb{F}^0 = (\mathcal{F}^0_t)_{t\geq 0}$ and $\mathbb{F}^1 = (\mathcal{F}^1_t)_{t\geq 0}$. We assume that the noise acting on the leader $B^Y$ and $Y_0$ are constructed on the space $(\Omega^0, \mathcal{F}^0,\mathbb{P}^0)$, while the randomness affecting the followers $(B^{i})_{i \in \mathbb{N}}$ and $(X_0^i)_{i \in \mathbb{N}}$ are constructed on $(\Omega^1, \mathcal{F}^1,\mathbb{P}^1)$. The product probability space is defined as $(\Omega,\mathcal{F},\mathbb{F},\mathbb{P})$, where $\Omega = \Omega^0 \times \Omega^1$, $(\mathcal{F},\mathbb{P})$ is the completion of $(\mathcal{F}^0 \otimes \mathcal{F}^1,\mathbb{P}^0 \otimes \mathbb{P}^1)$ and $\mathbb{F} = (\mathcal{F}_t)_{t\geq0}$ is the complete and right-continuous augmentation of $( \mathcal{F}_t^0 \otimes \mathcal{F}_t^1 )_{t\geq 0}$. Generic elements of $\Omega$ are denoted by $w = (w^0,w^1)$ with $w^0 \in \Omega^0$ and $w^1 \in \Omega^1$. 

Given this setup we then have the following results. For a random variable $X$ on $\Omega$ the mapping $\mathcal{L}^1(X) : \Omega^0 \ni w^0 \mapsto Law(X(w^0,\cdot))$ is a.s. well-defined under $\mathbb{P}^0$ and provides a conditional law of $X$ given $\mathcal{F}^0$ \cite[Lemma 2.4]{Carmona2018a}. Moreover, from \cite[Lemma 2.5]{Carmona2018a} we know that for an $\mathbb{R}^d$-valued process $(X_t)_{t\geq0}$ which is adapted to $\mathbb{F}$, the $\mathcal{P}(\mathbb{R}^d)$-valued process $(\mathcal{L}^1(X_t))_{t\geq0}$ is adapted to $\mathbb{F}^0$. If $(X_t)_{t\geq0}$ has continuous paths and for all $T\geq 0$, $\mathbb{E} [\sup_{t \leq T} |X_t|^2  ] < \infty$, then we can find a version of $\mathcal{L}^1(X_t)$ such that $(\mathcal{L}^1(X_t))_{t\geq0}$ is $\mathbb{F}^0$-adapted and has continuous paths. For the solution $(\bar{X}^N,\bar{Y} )$ of the conditional McKean-Vlasov SDE \eqref{eq:McKeanSDE} by \cite[Proposition 2.9]{Carmona2018a}, $\mathcal{L}^1(\bar{X}_t^{i,N})$ is a version of the conditional law of $\bar{X}_t^{i,N}$ given $(B^Y,Y_0)$, i.e. $Law(\bar{X}_t^{i,N}| \mathcal{F}_t^Y)$. We are now ready to prove the propagation of chaos result.

\begin{proof}[Proof of Theorem \ref{thm:PropChaos}]
Working under the probabilistic setup defined above we will use $\mathcal{L}^1(\bar{X}_t^{1,N})$ instead of $g_t$ in this proof. We will let $C$ be a generic constant, independent of $N$, which may change from line to line. By the definition of the strong solution, Jensen's inequality, and the Burkholder-Davis-Gundy inequality, we have
\begin{align*}
     \mathbb{E} \left[ \sup_{t \leq T} |X_t^{1,N} - \bar{X}^{1,N}_t |^2 \right]  \leq & C \mathbb{E} \left[ \int_0^T (b(t,X^{1,N}_t,Y_t^N,\mu_{X_t^N}) - b(t,\bar{X}^{1,N}_t,\bar{Y}_t,\mathcal{L}^1(\bar{X}_t^{1,N})))^2 dt \right]\\
     + &  C \mathbb{E} \left[ \int_0^T  (\sigma_X(t,X^{1,N}_t,Y_t^N,\mu_{X_t^N}) - \sigma_X(t,\bar{X}^{1,N}_t,\bar{Y}_t,\mathcal{L}^1(\bar{X}_t^{1,N})))^2  dt \right] .
\end{align*}
Let us only consider the first term in the integral, as the second can be treated in the exact same manner, due to the regularity assumptions on $\sigma_X$. By the triangle inequality
\begin{align*}
    &  \mathbb{E} \left[ \int_0^T (b(t,X^{1,N}_t,Y_t^N,\mu_{X_t^N}) - b(t,\bar{X}^{1,N}_t,\bar{Y}_t,\mathcal{L}^1(\bar{X}_t^{1,N})))^2  dt \right] \\
    & \leq C  \mathbb{E} \left[ \int_0^T (b(t,X^{1,N}_t,Y_t^N,\mu_{X_t^N}) - b(t,\bar{X}^{1,N}_t,\bar{Y}_t,\mu_{\bar{X}_t^N}))^2 dt \right] \\
    & + C \mathbb{E} \left[ \int_0^T (b(t,\bar{X}^{1,N}_t,\bar{Y}_t, \mu_{\bar{X}_t^N}) - b(t,\bar{X}^{1,N}_t,\bar{Y}_t,\mathcal{L}^1(\bar{X}_t^{1,N})))^2 dt \right] .
\end{align*}

From the definition of $W_2$ and since the pairs $(X^{i,N},\bar{X}^{i,N})_{1 \leq i\leq N}$ are identically distributed, we have
\begin{equation}
    \label{eq:Wass_to_part}
    \mathbb{E} \left[W_2(\mu_{X_t^N}, \mu_{\bar{X}_t^N})^2 \right] \leq \mathbb{E} \left[ \frac{1}{N} \sum_{i=1}^N |X_t^{i,N} - \bar{X}^{i,N}_t |^2 \right]  =\mathbb{E} \left[ |X_t^{1,N} - \bar{X}^{1,N}_t |^2 \right] .
\end{equation}
Due to Fubini's theorem, the global Lipschitz assumptions on the coefficient $b$ (see Assumption \ref{ass:Lipschitz}) and \eqref{eq:Wass_to_part} we therefore obtain
\begin{equation*}
    \begin{aligned}
     & \mathbb{E} \left[ \int_0^T (b(t,X^{1,N}_t,Y_t^N,\mu_{X_t^N}) - b(t,\bar{X}^{1,N}_t,\bar{Y}_t,\mathcal{L}^1(\bar{X}_t^{1,N})))^2  dt \right]  \\
     & \leq  C \int_0^T \mathbb{E} \left[ |X_t^{1,N} - \bar{X}^{1,N}_t |^2 + |Y_t^N - \bar{Y}_t|^2 + W_2(\mu_{\bar{X}_t^N},\mathcal{L}^1(\bar{X}_t^{1,N}))^2  \right] dt .         
    \end{aligned}
\end{equation*}

Following the exact same arguments for $\sigma_X$, we find that,
\begin{align*}
    \mathbb{E} \left[ \sup_{t \leq T} |X_t^{1,N} - \bar{X}^{1,N}_t |^2 \right] \leq  C \int_0^T \mathbb{E} \left[ |X_t^{1,N} - \bar{X}^{1,N}_t |^2 + |Y_t^N - \bar{Y}_t|^2 + W_2(\mu_{\bar{X}_t^N},\mathcal{L}^1(\bar{X}_t^{1,N}))^2  \right] dt ,
\end{align*}
and similarly due to the regularity assumptions on $c$, $u$ and $\sigma_Y$,
\begin{align*}
    \mathbb{E} \left[ \sup_{t \leq T} |Y_t^N - \bar{Y}_t|^2 \right] \leq  C \int_0^T \mathbb{E} \left[ |X_t^{1,N} - \bar{X}^{1,N}_t |^2 + |Y_t^N - \bar{Y}_t|^2 + W_2(\mu_{\bar{X}_t^N},\mathcal{L}^1(\bar{X}_t^{1,N}))^2  \right] dt .
\end{align*}
Combining these two inequalities, by the monotonicity of the expectation and Gr\"onwall's lemma, we have
\begin{equation*}
    \mathbb{E} \left[ \sup_{t \leq T} |X_t^{1,N} - \bar{X}^{1,N}_t |^2 + \sup_{t \leq T} |Y_t^N - \bar{Y}_t|^2 \right] \leq C \int_0^T \mathbb{E} \left[W_2(\mu_{\bar{X}_t^N},\mathcal{L}^1(\bar{X}_t^{1,N}))^2  \right] dt .
\end{equation*}
Since $\bar{X}^{i,N}$ are conditionally i.i.d. given $\mathcal{F}^0$ with common distribution $\mathcal{L}^1(\bar{X}_t^{i,N}) = \mathcal{L}^1(\bar{X}_t^{1,N})$ we have by Theorem \ref{thm:GlivenkoCantelli} that for any $t \leq T$,
\begin{equation*}
    \mathbb{P}^0 \left[ \lim_{N \to \infty}  \mathbb{E}^1 \left[ W_2(\mu_{\bar{X}_t^N},\mathcal{L}^1(\bar{X}_t^{1,N}))^2  \right] = 0 \right] = 1 .
\end{equation*}

By the triangle inequality and since $(\bar{X}_t^{i,N})_{1 \leq i \leq N}$ are conditionally i.d.d. given $\mathcal{F}_t^Y$,
\begin{equation*}
    \begin{aligned}
    \mathbb{E}^1 \left[ W_2(\mu_{\bar{X}_t^N},\mathcal{L}^1(\bar{X}_t^{1,N}))^2  \right] &\leq C \mathbb{E}^1 \left[ W_2(\mu_{\bar{X}_t^N},\delta_0)^2 \right]  + C \mathbb{E}^1 \left[ W_2(\delta_0,\mathcal{L}^1(\bar{X}_t^{1,N}))^2  \right] \\
    &\leq C \mathbb{E}^1 \left[ |\bar{X}_t^{1,N}|^2 \right] .
    \end{aligned}
\end{equation*}
Taking expectations with respect to the measure $\mathbb{P}^0$, we obtain by Proposition \ref{prop:well_posedness} that $\mathbb{E}^0 \left[ \mathbb{E}^1 \left[ |\bar{X}_t^{1,N}|^2 \right] \right] = \mathbb{E} \left[ |\bar{X}_t^{1,N}|^2 \right]  < \infty $ and therefore by the dominated convergence theorem 
\begin{equation}
    \label{eq:Lim_q2}
    \lim_{N \to \infty}  \mathbb{E} \left[ W_2(\mu_{\bar{X}_t^N},\mathcal{L}^1(\bar{X}_t^{1,N}))^2  \right] = 0 .
\end{equation}

If $\mathbb{E}\left[|X_0|^q \right] < \infty$ for $q> 4$, then from Proposition \ref{prop:well_posedness} we have that $\mathbb{E}\left[\sup_{t\leq T} |\bar{X}_t|^q \right] < \infty$ and we can use Theorem \ref{thm:ExplicitRates} to obtain
\begin{equation}
    \label{eq:Lim_q4}
    \mathbb{E} \left[ W_2(\mu_{\bar{X}_t^N},\mathcal{L}^1(\bar{X}_t^{1,N}))^2  \right] \leq C \epsilon_N,
\end{equation}
where $\epsilon_N$ is as in the statement of the theorem. Since, for any $s,t \in [0,T]$,
\begin{equation*}
    \left| \mathbb{E} \left[ W_2(\mu_{\bar{X}_t^N},\mathcal{L}^1(\bar{X}_t^{1,N}))^2  \right] - \mathbb{E} \left[ W_2(\mu_{\bar{X}_s^N},g_s)^2  \right]  \right| \leq C \mathbb{E} \left[ |\bar{X}^1_t - \bar{X}^1_s  |^2 \right]^{\frac{1}{2}} \leq C |t-s|^{\frac{1}{2}} ,
\end{equation*}
the expressions \eqref{eq:Lim_q2}, for $q\geq 2$, and \eqref{eq:Lim_q4}, for $q >4$, are uniformly equicontinuous on $[0,T]$. Together with the pointwise convergence, the Arzel\`a-Ascoli theorem implies uniform convergence, which completes the proof.
\end{proof}


\subsection{Proofs for Section \ref{sec:MFLcontrol}}
\label{sec:Appendix2}

\begin{proof}[Proof of Lemma \ref{lemma:ConvMcKean}]
    Similarly to the proof of Theorem \ref{thm:PropChaos} we have by the definition of the strong solution, Jensen's inequality, the Burkholder-Davis-Gundy inequality and the Lipschitz assumptions on the coefficients that 
    \begin{equation}
    \label{eq:Lem_bound_X_in_proof}
    \begin{aligned}
        \mathbb{E} \left[ \sup_{t \leq T} W_2^2 (g_{N,t}, g_t) \right] & \leq
        \mathbb{E} \left[ \sup_{t \leq T} \left| \bar{X}_{N,t} - \bar{X}_{t} \right|^2 \right]   \\  
        & \leq C \mathbb{E} \left[ \int_0^T \left| \bar{X}_{N,t} - \bar{X}_{t} \right|^2 + \left| \bar{Y}_{N,t} - \bar{Y}_{t} \right|^2 dt \right] ,
    \end{aligned}
    \end{equation}
    and
    \begin{equation}
    \label{eq:eq:Lem_bound_Y_in_proof}
    \begin{aligned}
        \mathbb{E} \left[ \sup_{t \leq T} \left| \bar{Y}_{N,t} - \bar{Y}_{t} \right|^2  \right] \leq & C \mathbb{E} \left[ \int_0^T\left| \bar{X}_{N,t} - \bar{X}_{t} \right|^2 + \left| \bar{Y}_{N,t} - \bar{Y}_{t} \right|^2 dt\right] \\
        &+  C \mathbb{E} \left[ \int_0^T \left|u_N(t,\bar{Y}_{N,t},g_{N,t})  - u(t,\bar{Y}_t,g_t) \right|^2 dt  \right] .
    \end{aligned}
    \end{equation}
    For the last term, by the definitions of $u$ and $u_N$, we have
    \begin{align*}
     \mathbb{E} \left[ \int_0^T \left|u_N(t,\bar{Y}_{N,t},g_{N,t})  - u(t,\bar{Y}_t,g_t) \right|^2 dt  \right]  \leq C \mathbb{E} \left[ \int_0^T  \left| \bar{Y}_{N,t} - \bar{Y}_{t} \right|^2 + \left| \bar{X}_{N,t} - \bar{X}_{t} \right|^2  dt + R_N  \right] ,
    \end{align*}
    where 
    \begin{equation*}
        R_N := \int_0^T  \left| f_N(\bar{Y}_{t},g_{t}) - f(\bar{Y}_{t},g_{t}) \right|^2 +  \left| (h_N(t)- h(t) ) f(\bar{Y}_t,g_t) \right|^2 dt .
    \end{equation*}
   From the uniform convergence of $f_N$ to $f$ and $h_N$ to $h$, we have that $\lim_{N \to \infty}  R_N  = 0$ almost surely. Since both $h_N$ and $f_N$ are bounded, by the dominated convergence theorem $\lim_{N \to \infty} \mathbb{E} \left[R_N  \right] = 0$. Combining this, \eqref{eq:Lem_bound_X_in_proof}, \eqref{eq:eq:Lem_bound_Y_in_proof}, and  using Gr\"onwall's lemma the claim follows.
    
\end{proof}
\begin{proof}[Proof of Lemma \ref{lem:liminf}]
    Since $r$ is uniformly continuous and bounded there exists a concave modulus of continuity $w_1$ such that 
    \begin{equation}
    \begin{aligned}
         \mathbb{E} \left[ \int_0^T \left| r(\bar{Y}_{N,t},g_{N,t}) - r(\bar{Y}_t,g_t) \right|  dt \right] \leq  T w_1 \left( \mathbb{E} \left[  \sup_{t \leq T} \left| \bar{Y}_{N,t} - \bar{Y}_t \right|^2 + \sup_{t \leq T} W_2^2(g_{N,t},g_t)    \right]^{\frac{1}{2}} \right) , \label{eq:liminfStep1}
    \end{aligned}
    \end{equation}
    where we have applied Jensen's inequality and used the fact that $w_1$ is increasing and positive. From Lemma \ref{lemma:ConvMcKean} we have that the right-hand side of the inequality converges to zero as $N \to \infty$. In general if $|a|,|b| \leq c$, then $w_2(|a-b|) := 4c|a-b| - |a-b|^2 \geq |a^2-b^2|$. By the triangle inequality and since $h_N$ and $f_N$ are both bounded, 
    \begin{equation*}
        \begin{aligned}
        & \mathbb{E} \left[ \int_0^T \left| (h_N(t)f_N(\bar{Y}_{N,t},g_{N,t}))^2 - (h_N(t)f(\bar{Y}_t,g_t))^2  \right| dt \right] \\
        & \leq C w_2 \left( \mathbb{E} \left[ \sup_{t \leq T} \left| f_N(\bar{Y}_{N,t},g_{N,t}) - f_N(\bar{Y}_t,g_t)   \right|   \right]   \right)  + C \mathbb{E} \left[ \int_0^T w_2\left(  \left| f_N(\bar{Y}_t,g_t) - f(\bar{Y}_t,g_t) \right| \right) dt \right] .
        \end{aligned}
    \end{equation*}  
    The first term in the last inequality converges to zero by the Lipschitz assumptions and Lemma \ref{lemma:ConvMcKean}. Since $f_N$ converges uniformly to $f$, 
    \begin{equation*}
    \int_0^T w_2\left(  \left| f_N(\bar{Y}_t,g_t) - f(\bar{Y}_t,g_t) \right| \right) dt \to 0,
    \end{equation*}
    almost surely. As $f_N$ is bounded, we have 
    \begin{equation}
        \label{eq:liminfStep2}
        \lim_{N \to \infty} \mathbb{E} \left[ \int_0^T \left| (h_N(t)f_N(\bar{Y}_t,g_t))^2 - (h_N(t)f(\bar{Y}_t,g_t))^2  \right| dt \right] = 0 ,
    \end{equation}
    by the dominated convergence theorem. From the uniform convergence of $h_N$ to $h$ and the dominated convergence theorem
    \begin{equation*}
        \lim_{N \to \infty} \mathbb{E} \left[ \int_0^T \left( h_N(t) f(\bar{Y}_t,g_t)   \right)^2 dt  \right] = \mathbb{E} \left[ \int_0^T \left( h(t) f(\bar{Y}_t,g_t)   \right)^2 dt  \right] ,
    \end{equation*}
    and therefore we obtain
    \begin{equation}
        \label{eq:liminfStep3}
        \liminf_{N \to \infty} \mathbb{E} \left[ \int_0^T (h_N(t)f(\bar{Y}_t,g_t))^2 dt \right] \geq \mathbb{E} \left[ \int_0^T (h(t)f(\bar{Y}_t,g_t))^2 dt \right] .
    \end{equation}
    We conclude by combining \eqref{eq:liminfStep1}, \eqref{eq:liminfStep2} and \eqref{eq:liminfStep3}.
\end{proof}

\begin{remark}
    Since $(\bar{Y}_t,g_t)$ are not on a compact space, to guarantee the existence of a concave modulus of continuity $w_1$, $r$ needs to be bounded.
\end{remark}

\begin{proof}[Proof of Theorem \ref{thm:GammaConvergence}]
    To prove $\Gamma$-convergence we need to show the liminf inequality and the existence of a recovery sequence. We have, similar to the derivation of \eqref{eq:liminfStep1}, that
    \begin{align*}
        & \left|  \mathbb{E} \left[ \int_0^T  r(Y_{N,t},\mu_{X_t^N}) - r(\bar{Y}_{N,t},g_{N,t}) dt \right]   \right| \leq  T w_1 \left( \mathbb{E} \left[  \sup_{t \leq T} \left| Y_{N,t} - \bar{Y}_{N,t} \right|^2 + \sup_{t \leq T} |X_t^{1,N} - \bar{X}_{N,t} |^2    \right]^{\frac{1}{2}} \right) .
    \end{align*}
    Similarly,
    \begin{align*}
        & \left|  \mathbb{E} \left[ \int_0^T  (h_N(t)f_N(Y_{N,t}, \mu_{X_t^N}))^2 - (h_N(t)f_N(\bar{Y}_{N,t},g_{N,t}))^2 dt \right]   \right| \\
        & \leq T w_2\left( C \mathbb{E} \left[  \sup_{t \leq T} \left| Y_{N,t} - \bar{Y}_{N,t} \right|^2 + \sup_{t \leq T} |X_t^{1,N} - \bar{X}_{N,t} |^2    \right]^{\frac{1}{2}} \right) .
    \end{align*}
    By the propagation of chaos result from Theorem \ref{thm:PropChaos}, which applies since the assumptions in Lemma \ref{lemma:ConvMcKean} ensure that the uniform bounds and Lipschitz constants are independent of the sequence, we therefore have
    \begin{equation}
        \label{eq:Flim}
        \lim_{N \to \infty} \left| J^N(u_N) - J(u_N)  \right| = 0.
    \end{equation}
    From Lemma \ref{lem:liminf} and \eqref{eq:Flim}  we therefore have the liminf inequality
    \begin{equation*}
        \liminf_{N \to \infty} J^N(u_N) = \liminf_{N \to \infty} \left(J^N(u_N) - J(u_N) \right) + \liminf_{N \to \infty} J(u_N) \geq J(u) .
    \end{equation*}
    To show the existence of a recovery sequence let, for any $u$, $u_N$ be the constant sequence $u_N = u$. From \eqref{eq:Flim} we therefore have the pointwise convergence $J(u) = \lim_{N \to \infty} J^N (u)$.
\end{proof}

\subsection{Proofs for Section \ref{sec:NumAlg}}
\label{sec:Appendix3}

\begin{proof}[Proof of Lemma \ref{lem:UniBoundsLinear}]
For an $\mathcal{M}(\mathbb{R}^d)$ valued solution $U$ to \eqref{eq:FrozenPDE} let $Z_\delta(s) = T_\delta U(s)$. Following the methodology of \cite{Kurtz1999}, since $T_\delta \varphi \in C_b^2 (\mathbb{R}^d)$ for all $\varphi \in C_b^2 (\mathbb{R}^d)$, we then have
\begin{equation*}
    \begin{aligned}
        \inner{Z_\delta(t)}{\varphi}_{L^2} = & \inner{Z_\delta(0)}{\varphi}_{L^2}  - \sum_{i=1}^d \int_0^t \inner{   \partial_{x_i} T_\delta ( b_{i,s} U(s) ) }{\varphi}_{L^2}   ds \\
        & + \frac{1}{2} \sum_{i,j=1}^d  \int_0^t \inner{  \partial_{x_i} \partial_{x_j} T_\delta(a_{ij,s} U(s)) }{\varphi}_{L^2}   ds   ,
    \end{aligned}
\end{equation*}
by the symmetry of $G_\delta$, the definition of the PDE \eqref{eq:FrozenPDE} and integration by parts. By the chain rule 
\begin{equation*}
    \begin{aligned}
        \inner{Z_\delta(t)}{\varphi}^2_{L^2} = & \inner{Z_\delta(0)}{\varphi}_{L^2}^2 - 2 \sum_{i=1}^d \int_0^t \inner{Z_\delta(s)}{\varphi}_{L^2} \inner{ \partial_{x_i} T_\delta ( b_{i,s} U(s) ) }{\varphi}_{L^2}   ds \\
        & +  \sum_{i,j=1}^d  \int_0^t \inner{Z_\delta(s)}{\varphi}_{L^2} \inner{  \partial_{x_i} \partial_{x_j} T_\delta(a_{ij,s} U(s)) }{\varphi}_{L^2}   ds.
    \end{aligned}
\end{equation*}
Summing over a complete, orthonormal basis $\varphi$  of $L^2(\mathbb{R}^d)$ and taking expectations, we obtain
\begin{equation*}
\begin{aligned}
    \mathbb{E} \left[ \norm{Z_\delta(t)}^2_{L^2}   \right] = &  \norm{Z_\delta(0)}^2_{L^2} - 2 \sum_{i=1}^d \mathbb{E} \left[ \int_0^t \inner{Z_\delta(s)}{\partial_{x_i} T_\delta ( b_{i,s} U(s) )}_{L^2}  ds   \right] \\
    & + \sum_{i,j=1}^d \mathbb{E} \left[ \int_0^t \inner{Z_\delta(s)}{ \partial_{x_i} \partial_{x_j} T_\delta(a_{ij,s} U(s))}_{L^2}   ds   \right] .   
\end{aligned}
\end{equation*}
From \cite[Lemma 3.3.]{Kurtz1999} and \cite[Lemma 3.2.]{Kurtz1999}, which hold under Assumption \ref{ass:Bound}, we find that, for some constant $C$,
\begin{equation*}
    \mathbb{E} \left[ \norm{Z_\delta(t)}^2_{L^2}   \right] \leq  \norm{Z_\delta(0)}^2_{L^2} + C \int_0^t \mathbb{E} \left[ \norm{T_{\delta}(|U_s| ) }_{L^2}^2 \right] .
\end{equation*}
Applying Gr\"onwall's lemma to the above inequality gives $    \mathbb{E} \left[ \norm{Z_\delta(t)}^2_{L^2}   \right] \leq \norm{Z_\delta(0)}^2_{L^2} e^{Ct} $. Taking a complete orthonormal basis of $L^2(\mathbb{R}^d) $, $\{ \phi_j \}$ such that $\phi_j \in C_b(\mathbb{R}^d)$, we have from the fact that $\lim_{\delta \to 0} G_\delta(x) = \delta(x)$ and Fatou's lemma that
\begin{equation*}
    \begin{aligned}
        \mathbb{E} \left[ \sum_j \inner{\phi_j}{U(t)}^2   \right] & \leq  \mathbb{E} \left[\liminf_{\delta \to 0} \sum_j \inner{\phi_j}{Z_\delta(t)}^2_{L^2} \right] \leq \liminf_{\delta \to 0} \mathbb{E} \left[ \sum_j \inner{\phi_j}{Z_\delta(t)}^2_{L^2} \right] \\
        &\leq \liminf_{\delta \to 0} \norm{Z_\delta(0)}^2_{L^2} e^{Ct} \leq \norm{U(0)}^2_{L^2} e^{Ct} < \infty ,         
    \end{aligned}
\end{equation*}
from which we can reach our conclusion.
\end{proof}

\end{document}